\documentclass[12pt]{article}
\usepackage{graphicx}
\usepackage{color}
\usepackage{amsmath}
\usepackage{amssymb}
\usepackage{amsthm}
\usepackage{mathtools}
\usepackage{float} 
\usepackage{caption}
\usepackage{chngcntr}

\usepackage{color}
\usepackage{url}
 \usepackage[top=2cm,	bottom=2cm,left=3cm,right=3cm]{geometry}
 \usepackage{youngtab}

\usepackage{hyperref}
\hypersetup{colorlinks=false, linkbordercolor={1 1 1}, citebordercolor={1 1 1}, urlbordercolor={1 1 1}} 
\usepackage[normalem]{ulem}

\newtheorem{theorem}{Theorem}
\newtheorem{proposition}[theorem]{Proposition}
\newtheorem{lemma}[theorem]{Lemma}
\newtheorem{corollary}[theorem]{Corollary}

\newtheorem{remark}[theorem]{Remark}
\def\paragraph#1{{\vskip3mm\noindent \bf #1 }}

\newcommand{\Z}{\mathbb Z}
\newcommand{\N}{\mathbb N}
\newcommand{\cA}{\mathcal A}
\newcommand{\cX}{\mathcal X}

\newcommand{\cZ}{\mathcal Z}

\newcommand{\cE}{\mathcal E}
\newcommand{\cS}{\mathcal S}

\newcommand{\cQ}{\mathcal Q}
\newcommand{\hcX}{\widehat{\cX}}

\newcommand{\cF}{\mathcal F}

\newcommand{\hmu}{\widehat{\mu}}
\newcommand{\hpi}{\widehat{\pi}}

\newcommand{\nusola}{\nu_\alpha}
\newcommand{\Zsola}{Z_\alpha}
\newcommand{\uvep}{\underline\vep}
\newcommand{\ux}{\underline x}

\newcommand{\tts}{\mathtt s}

\def\q{q}

\newcommand{\blue}[1]{{\color{blue} {#1}}}
\newcommand{\red}[1]{{\color{red} {#1}}}

\newcommand{\dd}{\mathrm{d}}

\def\sqr{\vcenter{
         \hrule height.1mm
         \hbox{\vrule width.1mm height2.2mm\kern2.18mm\vrule width.1mm}
         \hrule height.1mm}}                  
\def\square{\ifmmode\sqr\else{$\sqr$\vskip 3mm}\fi}

\newcommand{\one}{{\bf 1}\hskip-.5mm} 

\newcommand{\nn}{\nonumber}
\newcommand{\vep}{\varepsilon}
\DeclareMathOperator{\Palm}{Palm}

\definecolor{cmm}{rgb}{0,.6,0.4}
\definecolor{cmm'}{rgb}{.6,0,.4}

\title{\textsc{BBS invariant measures with independent soliton components}}
\date{\today}
\author{Pablo A. Ferrari, Davide Gabrielli}

\begin{document}

\maketitle
 
\begin{abstract}
The Box-Ball System (BBS) is a one-dimensional cellular automaton in $\{0,1\}^\Z$ introduced by Takahashi and Satsuma \cite{TS}, who also identified conserved sequences called \emph{solitons}. Integers are called boxes and a ball configuration indicates the boxes occupied by balls. For each integer $k\ge1$, a $k$-soliton consists of $k$ boxes occupied by balls and $k$ empty boxes (not necessarily consecutive). Ferrari, Nguyen, Rolla and Wang \cite{FNRW} define the $k$-slots of a configuration as the places where $k$-solitons can be inserted. Labeling the $k$-slots with integer numbers, they define the $k$-component of a configuration as the array $\{\zeta_k(j)\}_{j\in \mathbb Z}$ of  elements of $\Z_{\ge0}$ giving the number $\zeta_k(j)$ of $k$-solitons appended to $k$-slot $j\in \mathbb Z$. They also show that if the Palm transform of a translation invariant distribution $\mu$ has independent soliton components, then $\mu$ is invariant for the automaton. We show that for each $\lambda\in[0,1/2)$ the Palm transform of a product Bernoulli measure with parameter $\lambda$ has independent soliton components and that its $k$-component is a product measure of geometric random variables with parameter $1-q_k(\lambda)$, an explicit function of $\lambda$. The construction is used to describe a large family of invariant measures with independent components under the Palm transformation, including Markov measures.

\bigskip


\smallskip


\end{abstract}
\section{Introduction}
\label{s1}
Takahashi and Satsuma  \cite{TS}, referred to as TS in the sequel, introduced the \emph{Box-Ball System} (BBS), a cellular automaton describing the deterministic evolution of a finite number of balls on the infinite lattice $\mathbb Z$. A ball configuration $\eta$ is an element of $\{0,1\}^\Z$, where $\eta(z)=1$ indicates that there is a ball at box $z\in\Z$. A carrier visits successively boxes from left to right picking balls from occupied boxes and depositing one ball, if carried, at the current visited box, if empty. We denote by $T\eta$ the configuration obtained after the carrier has visited all boxes and $T^t\eta$ the configuration obtained after iterating this procedure $t$ times, for positive integer $t$. 

An example of the evolution of the Box-Ball dynamics is shown by the following example:
\begin{align}
\eta &\qquad 01101011010001111010000\nonumber\\
\textrm{Carrier\ Load}&\qquad 01212123232101234343210\label{carrier-load}\\
T\eta &\qquad 00010100101110000101111\nonumber
\end{align}
The configurations $\eta$ and $T\eta$ are identically $0$ outside the finite window shown. In the second line we write the number of balls which are transported by the carrier; we assume that the carrier is always empty outside of the window shown in the picture.

TS show the existence of \emph{basic sequences}, conserved quantities in the BBS called \emph{solitons} by Levine, Lyu and Pike \cite{levine-lyu-pike}. In the absence of other solitons, a $k$-soliton consists of $k$ successive occupied boxes followed by $k$ successive empty boxes. In this case, the $k$-soliton travels at speed $k$, because the carrier picks the $k$ balls and deposits them in the $k$ empty boxes of the soliton. Solitons with different speeds ``collide'' but still can be identified at collisions, see \S\ref{s4} for a description of the algorithm proposed by TS to identify solitons. A $k$-soliton consists always of $k$ occupied boxes and $k$ empty boxes which are however not necessarily consecutive; different solitons occupy disjoint sets of boxes.

A configuration of balls can be mapped to a walk that jumps one unit up at occupied boxes and one unit down at empty boxes \cite{CroydonKatoSasadaTsujimoto17} \cite{FNRW}. The \emph{excursions} of the walk are the pieces of configuration between two consecutive down \emph{records}. Walks coming from configurations with density of balls less than $\frac12$ have positive density of records, hence any box is either a record or belongs to a finite excursion. Ferrari, Nguyen, Rolla and Wang , referred to as FNRW in the sequel, introduce a soliton decomposition of each ball configuration. The soliton decomposition of an infinite configuration of balls is obtained applying the TS algorithm independently to each 
single finite excursion. See also \cite{FG} for a different soliton decomposition related to the trees underlying excursions.

A soliton decomposition of a ball configuration $\eta$ is a codification of $\eta$ in terms of the solitons and their spatial combinatorial arrangement. 
It consists of an infinite array $\zeta=\left(\zeta_k\right)_{k\in \mathbb N}$ where the $k$-\emph{component} $\zeta_k=\left(\zeta_k(j)\right)_{i\in \mathbb Z}$ has entries $\zeta_k(j)\in \mathbb Z_{\geq 0}$ representing \emph{the number of $k$ solitons appended to the $k$ slot number j}, for $j\in\Z$. The \emph{slots} are special lattice sites (to be determined by the configuration of particles) where the solitons can be appended. A $k$-slot is a slot where solitons up to order $k$ may be appended. Records are always slots of any order. We use the notation $D\eta:=\zeta$ and $D_k\eta:=\zeta_k$. 
FNRW proved that the $k$-component of the configuration $T\eta$ is a translation of the $k$-component of $\eta$, the amount translated depending on the $m$-components of $\eta$ for $m>k$.


Since the soliton decomposition is performed independently inside each excursion, it is convenient to introduce the finite array of components associated to one single excursion. This combinatorial object is called a \emph{slot diagram}.
The components of an infinite configuration of balls is obtained suitably joining the slot diagrams of its excursions.

Let $\mu$ be a translation invariant measure on the set of ball configurations with density less than 1/2 and call  $\hmu$ the record Palm measure of $\mu$, defined as the measure $\mu$ conditioned to have a record at the origin. FNRW show that if $\mu$ is translation invariant and $\hmu$ has independent $k$-components, then $\mu$ is invariant for the dynamics; we state their result in  Theorem \ref{t1} later. FNRW also study the asymptotic speed of solitons when the initial distribution of balls is translation invariant and ergodic. 

Let $\lambda\in[0,1)$ and call $\pi_\lambda$ the product measure of Bernoulli$(\lambda)$ random variables on the space $\{0,1\}^\Z$. Let $\widehat\pi_\lambda$ be its record Palm-measure
. In this paper we show that for $\lambda\in[0,\frac12)$, if $\eta$ is distributed according to $\widehat\pi_\lambda$, then the components $(D_k\eta)_{k\ge 1}$ are independent and each component $(D_k\eta(j))_{j\in\Z}$ consists of i.i.d.~Geometric random variables with parameter $1-q_k(\lambda)$,  computed later in Corollary \ref{c3}. We construct many other measures with independent components, being each component i.i.d.~Geometric random variables. A particular case is the distribution $\pi_{Q}$ of a stationary Markov chain with state space $\{0,1\}$ and transitions $Q(1,0)>Q(0,1)$, to guarantee that the density of $1$'s is less than $\frac12$; these are also nearest neighbor Ising-like measures with a negative external field.

The independence of components combined with Theorem \ref{t1} imply that $\pi_\lambda$ and the Ising-like measures are invariant for BBS. These facts were proven directly by Croydon, Kato, Sasada and Tsujimoto \cite{CroydonKatoSasadaTsujimoto17}, using reversibility of the carrier process illustrated in~\eqref{carrier-load}; see also \cite{FNRW}.

To prove the results just described we introduce two families of probability measures on the set of finite excursions.
The first family, contain measures called $\nu_\alpha$ indexed by $\alpha=(\alpha_k)_{k\ge1}$, a collection of parameters in $[0,1)$ satisfying a summability condition. Under $\nu_\alpha$ each excursion has weight $\prod_{k\ge1}\alpha_k^{n_k}$, where $n_k$ is the number of $k$-solitons in the excursion. The second family, called $\varphi_q$ is indexed by parameters $q_k\in[0,1)$, $k\ge1$, also satisfying some summability condition. Conditioning on the components $m> k$ of the slot diagram of the random excursion with law $\varphi_q$, the distribution of the $k$-component is a product of $s_k$ geometric distributions with  mean $q_k/(1-q_k)$, where $s_k$ is the number of $k$-slots determined by the $m$-components, for  $m$ bigger than $k$.
Theorem \ref{teonuovo}, one of the main results of this paper, shows a bijection between those two families with an explicit relation between $\alpha$ and $q$, see \eqref{ppc33} later. Under suitable assumptions, the resulting random excursion has finite mean length. 

We then consider a sequence of i.i.d.~excursions with law $\nu_\alpha$ and finite expected  excursion length and construct a ball configuration $\eta$ by putting a record at the origin and concatenating the excursions separated by records; the  distribution of $\eta$ is a record translation invariant measure. We show that the components $(D_k\eta)_{k\ge1}$ are independent and that $(D_k\eta(j))_{j\in\Z}$ are i.i.d.~Geometric random variables with mean $q_k/(1-q_k)$, where $q$ is a function of  $\alpha$. Using the inverse-Palm transformation, we obtain a translation invariant and $T$-invariant measure. The $T$-invariance is deduced from the independence of the components, as explained before. We show that product of Bernoulli and Ising-like measures conditioned to have a record at the origin have i.i.d.~excursions with distribution $\nu_\alpha$ for suitable $\alpha$, which in turn implies that have independent components and are $T$-invariant. 

\smallskip

The paper is organized as follows. 

In Section \ref{s4} we introduce notation, illustrate the soliton decomposition, define the slot diagrams and show that
they are in bijection with excursions.

In Section \ref{RE} we introduce the families of probability measures on the set of excursions parametrized by an infinite collection of parameters and show in Theorem \ref{teonuovo} that these are two different parametrization of the same family of probability measures with a non trivial relationship between the two families of parameters.

In Section \ref{imb} we obtain $T$-invariant measures concatenating i.i.d.~random excursions with distribution $\nu_\alpha$. This is obtained in Theorem  \ref{t2} by the combination of Theorems \ref{teo9} and \ref{t1}. These are the remaining main results of the paper.

\section{Excursions, solitons  and slot diagrams}\label{s4}

In this Section we define \emph{excursions}, describe a variant of the Takahashi-Satsuma Algorithm in \cite {TS} to identify solitons in the excursions and  call \emph{slot diagram} the FNRW soliton decomposition of an excursion.

\smallskip

A configuration of balls is an element $\eta\in \{0,1\}^\mathbb Z$, where for each \emph{box} $y\in\Z$, $\eta(y)=1$ means that there is a \emph{ball} at box $y$, otherwise $\eta(y)=0$ means $y$ is empty. In this Section we consider configurations with a finite number of balls. 

Map a ball configuration $\eta$ to a  walk $\xi=W\eta\in\Z^\Z$ defined up to a global additive constant by
\begin{align}
\label{x11}
\xi(z)-\xi(z-1)=2\eta(z)-1.
\end{align}
We fix the constant by choosing $\xi(0)=0$.
The configuration of balls is completely determined by the walk and if $\xi=W\eta$ we write also $\eta=W^{-1}\xi$. 

We call $z\in \mathbb Z$ a \emph{record} for $\xi$ if $\xi(z)<\xi(z')$ for any $z'<z$. This depends just on $\eta$ as $\xi(z)-\xi(z') = \sum_{y=z'+1}^{z}(2\eta(y)-1)$ and we can therefore say equivalently that $z$ is a record for the configuration $\eta$.

\paragraph{Excursions} 
We introduce the set $\mathcal E$ of \emph{finite soft excursions}.
An element $\vep\in \mathcal E$ is a finite walk which starts and ends at
zero, it is always non-negative and it has length $2n(\vep)$. More precisely
$\vep=\big( \vep(0),\dots, \vep(2n(\vep))\big)$ with the constraints $|\vep(z)-\vep(z-1)|=1$, $\vep(z)\ge 0$ for $0\le z\le n(\vep)$ and
$\vep(0)=\vep(2n(\vep))=0$.
The empty excursion $\emptyset$ is also an element of $\cE$ with $n(\emptyset)=0$.
We call $\mathcal E_n$ the set of soft finite excursions of length $2n$, hence
$\mathcal E=\cup_{n=0}^{+\infty}\mathcal E_n$. It is well known \cite{MR1676282} that the number of excursions of length $2n$ is given by
\begin{equation}\label{catalan}
|\mathcal E_n|=\frac{1}{n+1}\binom{2n}{n}\,;
\end{equation}
the right hand side is the Catalan number $C_n$. In the following we call soft excursions simply excursions.

The underlying configuration of balls of an excursion $\vep$ is called $W^{-1}\vep$ and is defined by
$$
W^{-1}\vep(z):=\frac{\vep(z)-\vep(z-1)+1}{2}\,,\qquad z=1,\dots, 2n(\vep)\,.
$$ 
This is a configuration of balls restricted to the interval $[1,2n(\vep)]$ but we can naturally extend it to a configuration on the whole axis $\mathbb Z$ just considering empty all the remaining boxes. This corresponds to extend the excursion to an infinite walk adding to the left and to the right just downward oriented steps. 

We use the same notation both for configuration of balls/walks restricted to a finite interval and for  configuration of balls/walks on the whole $\mathbb Z$ axis. The exact meaning will be clear from the context. We call an excursion both the walk $\vep$ and the corresponding configuration of balls $W^{-1}\vep$, since they are bijectively related.

\paragraph{Takahashi-Satsuma Identification of solitons}
We describe a variant of the Takahashi-Satsuma algorithm \cite{TS} to identify the solitons of a finite ball configuration $\eta$. The empty configuration $\eta(z)\equiv 0$ has no solitons. Assume $\eta$ is nonempty. A \emph{run} of $\eta$ is any segment $[z,y]$ with $-\infty\le z\le y\le \infty$ such that $\eta(z) = \eta(z')$, for $z'\in[z,y]$, $\eta(z-1)\neq \eta(z)$ if $z>-\infty$ and $\eta(y)\neq\eta(y+1)$ if $y<\infty$. The ball configuration underlying an excursion (considered on the whole lattice) has two semi-infinite runs and a finite number of finite runs. The algorithm is the following:

If there are finite runs in the configuration, do:
\begin{enumerate}
	\item
	Let $k$ be the size of the smallest run in the configuration.
	Select the leftmost run of size $k$. 
        Set the restriction of $\eta$ to the $k$ boxes of this run and the first $k$ boxes of the
successive run as a $k$-soliton.

	\item Ignore the boxes belonging to already identified solitons, update the runs of the remaining configuration and go to 1.
        \end{enumerate}
    
	
 \noindent\begin{minipage}{\linewidth}\begin{center}	
	\includegraphics[width=.8\textwidth ]{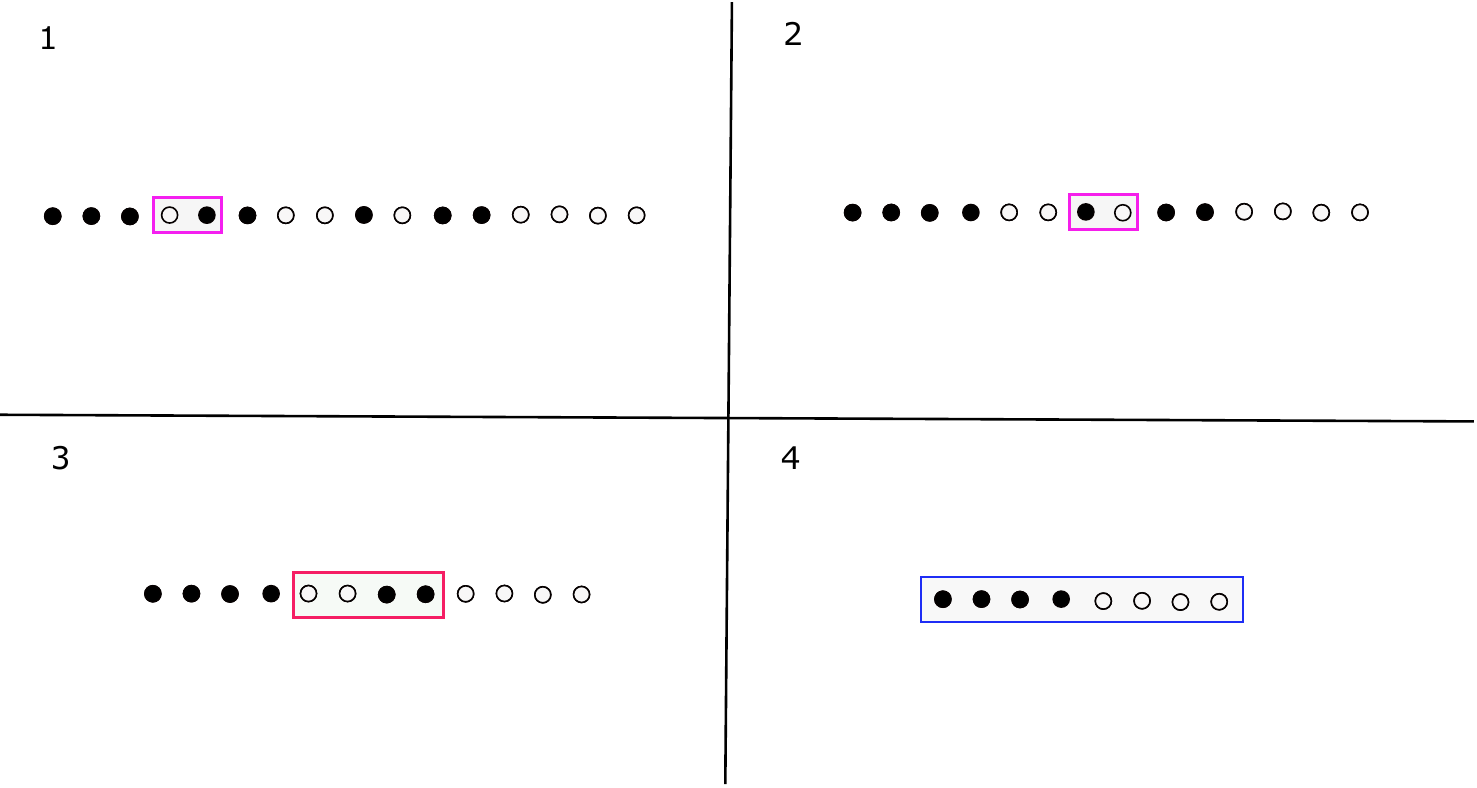}
	
	\captionof{figure}{The Takahashi-Satsuma algorithm applied to the finite configuration in quadrant 1 which is indeed an excursion. Identified solitons are surrounded by rectangles of different colors (violet for $1$-solitons, red for $2$-solitons, blue for $4$-solitons). The algorithm stops after 4 iterations. \label{TS-algo}	 }
\end{center}\end{minipage}	
	
 \noindent\begin{minipage}{\linewidth}\begin{center}	
	
	\includegraphics{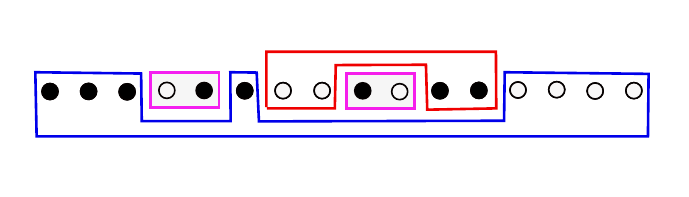}
	
	\captionof{figure}{The final decomposition into solitons of the configuration of Fig.~\ref{TS-algo}. Balls and boxes belonging to the same soliton are surrounded by colored lines. The lines are violet for $1$-solitons, red for $2$-solitons, blue for $4$-solitons. }
	\label{soli-deco}	
      \end{center}\end{minipage}

For a $k$-soliton $\gamma$ we call \emph{support} of $\gamma$, denoted by $\{\gamma\}\subset\Z$, the union of two sets of boxes: the \emph{head} $\{h_0(\gamma),\dots,h_{k-1}(\gamma)\}$ and the \emph{tail} $\{t_0(\gamma),\dots,t_{k-1}(\gamma)\}$, satisfying $\eta(h_i)=1$ and $\eta(t_i)=0$ and $h_i(\gamma)<h_{i+1}(\gamma)$, $t_i(\gamma)<t_{i+1}(\gamma)$ for $i=0,\dots,k-2$. Either $h_i(\gamma)<t_j(\gamma)$ for all $i,j$ or  $t_j(\gamma)<h_i(\gamma)$ for all $i,j$. We denote by $\Gamma_k\eta$ the set of $k$-solitons of $\eta$. When $\eta$ has infinitely many records to the right and left of the origin, every box in $\Z$ is either a record or belongs to $\{\gamma\}$ for some $k$-soliton $\gamma$, for some $k\ge 1$.

An example of the application of this algorithm is illustrated in Fig.~\ref{TS-algo}. The final decomposition into solitons is illustrated in Fig.~\ref{soli-deco}

\paragraph{Slots} Given an excursion $\vep$, a box $z$ is a \emph{$k$-slot} if either $z$ is Record 0 $z=0$ or $z \in\{h_i(\gamma), t_i(\gamma)\}$ for some $i\ge k$, some  $\gamma\in\Gamma_m\eta$ for some $m>k$. Let $S_k\eta$ be the set of $k$-slots of $\eta$. We have $S_{k+1}\eta\subseteq S_k\eta$.

Enumerate the $k$-slots setting  $\tts_k(\eta,0):=0$, that is,  $k$-slot 0 is at record 0 for all $k$, and
\begin{align}
\hbox{$\tts_k(\eta,j):=$ position of the $j$-th $k$-slot, counting from $k$-slot 0}.
\end{align}
We show in Figures \ref{slot-paper} and \ref{slot-paper2} an example of identifications of the slots using the sample configuration $\eta$ of Fig.~\ref{TS-algo}.

	
	 \noindent\begin{minipage}{\linewidth}\begin{center}
	\includegraphics{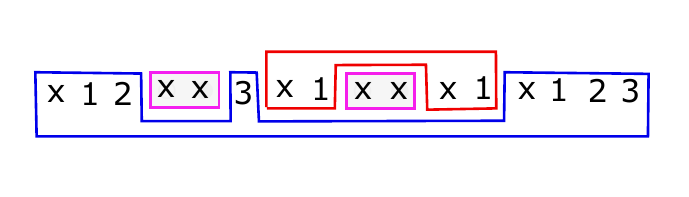}
	
	\captionof{figure}{Slots associated to our sample configuration of Fig.~\ref{TS-algo}. To each box we associate the number of the maximal slot. The symbol $\times$ means that the box is not a slot for any $k\geq1$. A box with number $m$ is a $k$-slot for each $k\leq m$. 	\label{slot-paper}}
	
\end{center}\end{minipage} 


	 \noindent\begin{minipage}{\linewidth}\begin{center}

\includegraphics{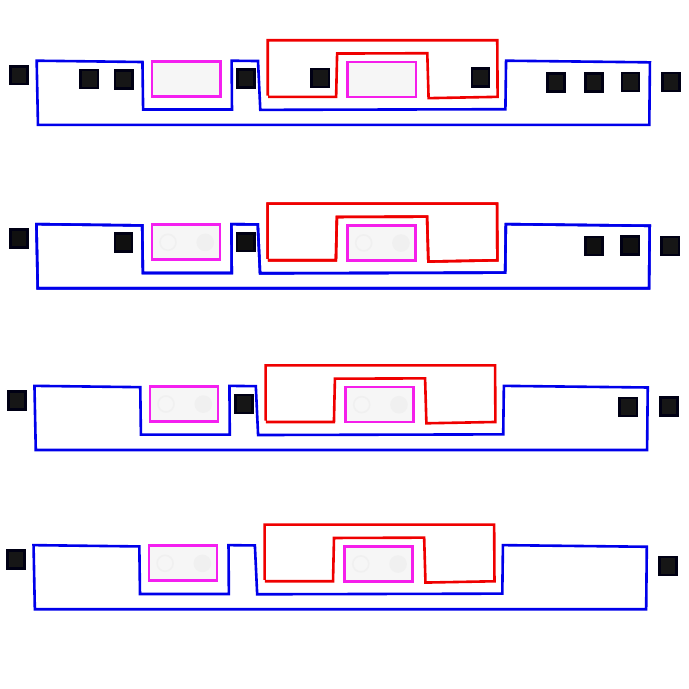}

\captionof{figure}{From the top to the bottom we represents respectively the sets $S_1\eta$, $S_2\eta$, $S_3\eta$ and $S_4\eta$. The configuration $\eta$ is our sample configuration of Fig.~\ref{TS-algo}. Boxes belonging to the sets are marked by a $\blacksquare$. We marked also the slots associated to the records on the left and on the right of the finite configuration (which is indeed an excursion). The origin is the leftmost black square. Numbers to the slots are given starting counting from this slot. \label{slot-paper2}}
	
\end{center}\end{minipage} 

We get all slots together in Fig.~\ref{slot-enumerated-3}.

\noindent\begin{minipage}{\linewidth}\begin{center}

\includegraphics{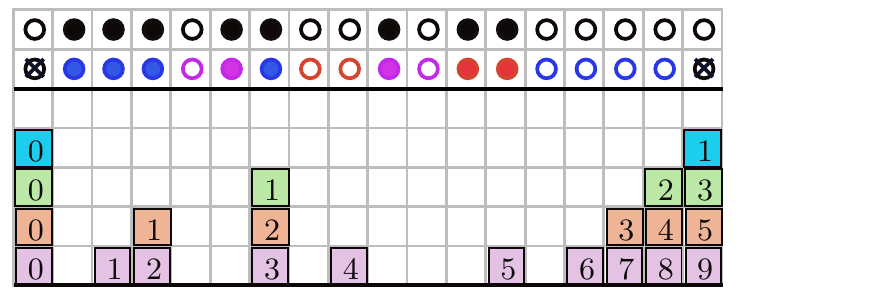}

\captionof{figure}{Slot enumeration. The configuration $\eta$ is our sample configuration of Fig.~\ref{TS-algo}, an excursion between two records. In the second line the solitons have been identified and colored: a blue $4$-soliton, a red 2-soliton and two purple 1-solitons. Below, colored square boxes identify $k$-slots for $k=1,2,3,4$. Records are $k$-slots for all $k$ but we have only depicted until $k=4$ as there are no solitons bigger than 4 and hence no slots bigger than 4 besides the records. There is a 4-soliton appendend at 4-slot 0, a 2-soliton appended at 2-slot 2 and two 1-solitons appended at 1-slots 2 and 4, respectively. The slots associated to the record located to the right of the excursion do not strictly belong to the excursion; they are still depicted because they are needed to identify the solitons appended to the slots to the left of it.  \label{slot-enumerated-3}}
	
\end{center}\end{minipage}

\paragraph{Soliton decomposition of ball configurations \cite{FNRW}}   We say that a $k$-soliton $\gamma$ is \emph{appended} to $k$-slot $j$ of $\eta$ if its support is strictly included in the open integer interval with extremes in the $k$-slots $j$ and $j+1$:
\begin{align}
\{\gamma\}\subset \big( \tts_k(\eta,j), \tts_k(\eta,j+1)\big).\label{appended1}
\end{align}
Any finite number of $k$-solitons may be appended to a single $k$-slot.
Define
\begin{align}
  \label{appended}
\zeta_k(j):= \#\{\gamma\in\Gamma_k\eta: \gamma\hbox{ is appended to $k$-slot }j\}.
\end{align}

Consider the example of Fig.~\ref{slot-paper2}. Starting from the bottom we have that the blue $4$-soliton is between $\tts_4(\eta,0)$ and $\tts_4(\eta,1)$ so that it is appended to the $4$-slot number $0$ and $\zeta_4(0)=1$; the red $2$-soliton is between $\tts_2(\eta,2)$ and $\tts_2(\eta,3)$ so that it is appended to the $2$-slot number $2$ and $\zeta_2(2)=1$; the violet $2$-solitons are respectively between $\tts_1(\eta,2)$ and $\tts_1(\eta,3)$ and $\tts_1(\eta,4)$ and $\tts_1(\eta,5)$ so that the leftmost 1-soliton is appended to the $1$-slot number $2$ while the rightmost 1-soliton is appended to the $1$-slot number $4$ and therefore we have $\zeta_1(2)=1$ and $\zeta_1(4)=1$. All the remaining $\zeta$'s are identically zero.
See also Fig.~\ref{slot-enumerated-3}.

\paragraph{Slot diagrams} A slot diagram is a combinatorial object which encodes the components of a single excursion.

We start giving a formal definition. A \emph{Slot Diagram} is a family
$x= (x_k)_{k\ge1}$ of vectors $x_k=(x_k(0),\dots,x_k({s_k-1}))$ with $s_k\in\N$ and $x_k(j)\in\Z_{\ge 0}$, satisfying the following conditions:
denoting by $|x_k|:=  x_k(0)+\dots+x_k({s_k-1})$, we have
\begin{align}
1)&\qquad M(x):= \max\{k: x_k(0)>0\} <\infty\,, \label{sk0}\\
2)& \qquad s_\ell=1, \hbox{ for } \ell\ge M(x) \hbox{ and } x_\ell(0)=0 \hbox{ for } \ell> M(x)\,,\label{sk}\\
3)& \qquad s_k= 1+ \sum_{\ell>k} 2(\ell-k)|x_\ell|\,. \label{sk2}
\end{align}
The complete structure of a slot diagram is determined by the finite collection of vectors $(x_k)_{1\le k\leq M}$ but for notational convenience we consider also the indices $k>M=M(x)$. An example of a slot diagram is the following
\begin{align}\label{slotex2}
& k \qquad \mapsto \qquad x_k\nonumber \\
& 4 \qquad \mapsto \qquad (1)\nonumber \\
& 3 \qquad \mapsto \qquad (0,0,0) \\
& 2 \qquad \mapsto \qquad (0,0,1,0,0)\nonumber \\
& 1 \qquad \mapsto \qquad (0,0,1,0,1,0,0,0,0)\nonumber
\end{align}
In this case we have $M=4$, $s_4=1$, $s_3=3$, $s_2=5$ and $s_1=9$. For any $k>4$ we have $s_k=1$ and $x_k(0)=0$ and therefore the slot diagram is completely determined by the finite diagram \eqref{slotex2}.

Let $\cS$ be the set of slot diagrams. We have that $\cE$ is in bijection with $\cS$ so that a slot diagram completely codifies an excursion. We now construct the map $\vep\mapsto x[\vep]$ and its inverse $x\mapsto\vep[x]$ (see \cite{FNRW} and \cite{FG} for more details).

\paragraph{Construction of $x[\vep]$} Consider an excursion $\vep$. If the excursion is empty then the slot diagram is defined as $s_k\equiv 1$ and $x_k(0)\equiv 0$. If $\vep$ is not empty, then let $M$ be the maximal soliton size in $\vep$ and define $s_\ell=1$ for $\ell\ge M$, $x_\ell(0)=0$ for $\ell>M$ and set $x_M(0)=$ number of $M$-solitons in the excursion. Assume we have set $x_{k+1},\dots, x_M$. Use \eqref{sk2} to define the number of $k$-slots $s_k$ and set $x_k(j)=$ number of $k$-solitons appended to $k$-slot $j$ in the excursion. Iterate for $k=M-1,\dots,1$.

In short, considering the excursion $\vep$ as an infinite walk we have that $s_k-1$ is the number of $k$-slots of the excursion which are not records and $x_k(j)$ is the number of $k$-solitons appended to the $k$-slot number $j$.
For example \eqref{slotex2} is the slot diagram associated to the excursion corresponding to the ball configuration in Fig.~\ref{TS-algo} and Fig.~\ref{slot-enumerated-3}.

\paragraph{Construction of $\vep[x]$} Given a configuration $\eta$ with no $\ell$-solitons for $\ell<k$, define $I_{k,j}$ the operator that insert a $k$-soliton at $k$-slot $j$ of $\eta$, as follows. Denote by $u=\tts_k(\eta,j)$ the position of $k$-slot $j$ in $\eta$ and
\begin{align}
I_{k,j}\eta(z) =
\begin{cases}
\eta(z) &\hbox{if } z\le u\\
1-\eta(u)&\hbox{if } u<z\le u+k\\
\eta(u)&\hbox{if } u+k<z\le u+2k\\
\eta(z-2k)&\hbox{if } u+2k<z.
\end{cases}
\end{align}
Denote by $I_{k,j}^n$ the $n$-th iteration of $I_{k,j}$, which corresponds to insert $n$ $k$-solitons one after the other on the same slot $j$. When $n=0$ we just have the identity, meaning that no $k$-soliton is inserted at slot $j$.

Denoting $M:=M(x)$, define
\begin{align}
\eta_\ell &\equiv 0 \quad\hbox{for } \ell>M, \quad \hbox{and iteratively,}
\nn\\
\eta_k &:= I_{k,0}^{x_k(0)}\dots I_{k,s_k-1}^{x_k(s_k-1)} \eta_{k+1} , \quad\hbox{for } k=M,\dots,1.\label{d88} \\
\vep[x] &:= W\eta_1.\nn
\end{align}
Observe that the number $n_k$ of $k$-solitons in the excursion $\vep[x]$ coincides with the sum over $j$ of $x_k(j)$:
\begin{align}
\label{nkxk}
n_k(\vep[x]) = \sum_{j=0}^{s_k-1}x_k(j) =|x_k|.
\end{align}
\emph{Example}. Consider the following slot diagram $x$:
\begin{align}
\label{slotex}
x_\ell&= (0), \quad\hbox{for }\ell>3 \nonumber\\
x_3  &= (2)\nonumber\\
x_2 &= (0,0,1,0,0)\\
x_1 &= (3,0,4,1,0,0,0,0,2,0,1)\nonumber
\end{align}
that is, $M=3$, $s_k=1$ for $k\geq 3$, $s_2=5$ and $s_1=11$.

In this example the algorithm works as follows. Active $k$-slots are red and $k$-solitons being appended at each step are blue.

{\tt \red{0}}\qquad\qquad\qquad (record 0 = \red{$k$-slot} 0 for all $k$)\\
{  \tt\red{0}\blue{111000111000}}\qquad\qquad (attach 2 \blue{3-soliton} to \red{3-slot} 0)  $I^2_{3,0}$\\
{ \tt \red{0}11\red{1}00\red{0}\blue{1100}11\red{1}00\red{0}  }\qquad\qquad (attach 1 \blue{2-soliton} to \red{2-slot} 2) $I^1_{2,2}$\\
{ \tt \red{0}\blue{101010}1\red{11}0\red{00}1\red{1}0\red{0}1\red{11}0\red{00} }\qquad\qquad (attach 3 \blue{1-soliton} to \red{1-slot} 0) $I^3_{1,0}$\\
{ \tt \red{0}1010101\red{11}\blue{01010101}0\red{00}1\red{1}0\red{0}1\red{11}0\red{00} }\qquad\qquad (attach 4 1-soliton to 1-slot  2) $I^4_{1,2}$\\
{ \tt \red{0}1010101\red{11}010101010\red{0}\blue{10}\red{0}1\red{1}0\red{0}1\red{11}0\red{00}}\qquad\qquad (attach 1 1-soliton to 1-slot 3)  $I^1_{1,3}$\\
{ \tt \red{0}1010101\red{11}010101010\red{0}{10}\red{0}1\red{1}0\red{0}1\red{11}\blue{0101}0\red{00}}\qquad(attach 2 1-solitons to 1-slot 8)  $I^2_{1,8}$\\
{ \tt \red{0}1010101\red{11}010101010\red{0}{10}\red{0}1\red{1}0\red{0}1\red{11}{0101}0\red{00}\blue{10}}\qquad(attach 1 1-soliton to 1-slot 10) $I^1_{1,10}$

The resulting excursion is given by
\begin{align}
\vep[x] &= W\left(I^1_{1,10}I^2_{1,8}I^1_{1,3}I^4_{1,2}I^3_{1,0}I^1_{2,2}I^2_{3,0}\eta_4\right) \nn\\
&= W\left(\hbox{\tt...\blue{101010}\red{111}\blue{01010101}\red{0}\blue{10}\red{00}{\color{green}1100}\red{111}\blue{0101}\red{000}\blue{10}...}\right)
\nn
\end{align}
where the dots represent records and we have painted blue, green and red the 1-, 2- and 3-solitons, respectively.
Record 0 is the dot preceding the leftmost 1 and record 1 is the dot following the rightmost 0. Here we start with the empty excursion $\eta_4$ because $M=3$.

\section{Random excursions}
\label{RE}

We introduce two natural families of probability measures on the set of excursions $\mathcal E$ depending on two collections of parameters $\alpha$ and $q$. The main result of this section is that the two families coincide with a non trivial relationship between the parameters. 

For $p\in(0,1]$ we say that a random variable $Y$ is Geometric$(p)$ when
\begin{align}
\label{geomm}
P(Y=j)= p (1-p)^j,\;j\ge 0;\quad EY= \frac {1-p}{p}.
\end{align}
with the convention $0^0=1$.

\subsection{Probability measures on excursions}
\label{ism}

\paragraph{First family}
For each excursion $\vep\in\cE$ define
\begin{align}
n_k(\vep) := \hbox{number of $k$-solitons in $\vep$},
\end{align}
where this number is given by the Takahashi-Satsuma algorithm in \S2 applied to $\vep$.

Let $\alpha = (\alpha_k)_{k\ge 1}$ be a family of parameters with $\alpha_k\in[0,1)$, define
\begin{align}
\label{aaa}
\Zsola:=\textstyle{\sum_{\vep\in\cE} \prod_{k\ge1} \alpha_k^{n_k(\vep)}}
\end{align}
and call
\begin{equation}\label{aaaa}
\mathcal A:=\left\{\alpha \,:\, Z_\alpha <+\infty \right\}\,.
\end{equation}
For $\alpha\in \mathcal A$ define the measure $\nusola$ on $\cE$ by
\begin{align}
\label{nusola}
\nusola(\vep) := \frac1{\Zsola}\textstyle{\prod_{k\ge1} \alpha_k^{n_k(\vep)}}\,,
\end{align}
here again we use the convention $0^0=1$ so that if $\alpha_k=0$ then the measure $\nusola$ gives full measure to excursions without $k$-solitons.
Note that by \eqref{nkxk} we can write \eqref{nusola} in terms of the slot diagram of $\vep$ by
\begin{align}
\label{nusolax}
\nusola(\vep) = \frac1{\Zsola}\textstyle{\prod_{k\ge1} \alpha_k^{|x_k[\vep]|}}\,.
\end{align}
We denote the mean number of $k$-solitons per excursion by
\begin{align}\label{ee46}
\rho_k(\alpha) := \sum_{\vep\in\cE} n_k(\vep)\, \nu_\alpha(\vep),
\end{align}
and therefore the mean excursion size is
\begin{align}
\label{ee45}
\sum_{k\ge1}2k\,\rho_k(\alpha)=\frac1{\Zsola}\sum_{\vep\in\cE}\left[\Bigl(\sum_{j\ge1} 2j\,n_j(\vep)\Bigr) \prod_{k\ge1} \alpha_k^{n_k(\vep)}\right]\,.
\end{align}
We call $\cA^+$ the set of $\alpha$ such that the mean excursion size under $\nu_\alpha$ is finite: 
\begin{equation}
\mathcal A^+:=\left\{\alpha\,:\, \textstyle{\sum_{k\ge1}2k\rho_k(\alpha)}<+\infty \right\}\,.\label{a+}
\end{equation}
By definition we have $\mathcal A^+\subseteq \mathcal A$.

\paragraph{Second family} 
Let $q=(q_k)_{k\geq 1}$ be a family of parameters with $q_k\in [0,1)$ and introduce the sets
\begin{align}\label{sumq}
\mathcal Q&:=\left\{q\,:\, \textstyle{\sum_{k\ge1}}\,q_k<+\infty\right\}\,,\\
\mathcal Q^+&:=\left\{q\,:\, \textstyle{\sum_{k\ge1}}\,kq_k<+\infty\right\}\,.\label{sumq+}
\end{align}
For $q\in \mathcal Q$ consider the probability measure $\varphi_q$ on $\mathcal E$ defined by
\begin{align}
\label{ss24}
\varphi_q(\vep) := \textstyle{\prod_{k\ge 1}}\,q_k^{|x_k[\vep]|} (1-q_k)^{s_k(x[\vep])}.
\end{align}
The fact that \eqref{ss24} is a probability measure on $\mathcal E$ when $q\in \mathcal Q$ is a consequence of the following argument.
Writing $x=x[\vep]$ and denoting
$x_k^\infty =(x_k, x_{k+1}, \dots)$,
formula \eqref{ss24} is equivalent to the following three formulas (with the convention $q_0=1$ to take care of the empty excursion), which give a recipe to construct/simulate the random slot diagram of an excursion with distribution \eqref{ss24}
\begin{align}
\varphi_q\left(M(x)=m\right)&
= q_m\,\prod_{\ell>m}\big(1-q_\ell),    \quad m\ge 0,                        
\label{ss36}\\[2mm]
\varphi_q\big(x_m(0)\big|M(x)=m\big)&= q_m^{|x_m(0)|-1} (1-q_m),  \label{ss37}\\[2mm]
\varphi_q\big(x_k\big|x_{k+1}^{\infty}\big)&= q_k^{|x_k|} (1-q_k)^{s_k(x)}, \label{ss38}
\end{align}
where we abuse notation writing $x_m$ as ``the set of slot diagrams $y$ such that $y_m=x_m$'', and so on.
Then, to construct a slot diagram with law $\varphi_q$, first choose a maximal soliton-size $m$ with probability \eqref{ss36}. This is a probability on $\mathbb Z_{\geq 0}$ since $q\in \mathcal Q$. Then use \eqref{ss37} to determine the number of maximal solitons $x_m(0)$ (a Geometric$(1-q_m)$ random variable conditioned to be strictly positive). Finally we use \eqref{ss38} to construct iteratively the lower components. Under  the measure $\varphi_q$ and conditioned on $x_{k+1}^{\infty}$, the variables $\left(x_k(0), \dots ,x_k(s_k-1)\right)$ are i.i.d.~Geometric$(1-\q_k)$.

\subsection{Equivalence of measures}

Given the parameters $\alpha$ and $q$ we define the transformation $q=q(\alpha)$ by
\begin{align}
q_1&:=\alpha_1\quad \text{ and }\quad
q_k:=\frac{\alpha_k}{\prod_{j=1}^{k-1}(1-\q_j)^{2(k-j)}}\,,\text{ for }k\ge 2 , \quad \label{ppc33}
\end{align}
and $\alpha=\alpha(q)$ by
\begin{equation}
\alpha_k:=\q_k\prod_{\ell=1}^{k-1}(1-\q_\ell)^{2(k-\ell)}\,, \quad \text{ for }k\ge 1.\label{ppc3d}\\
\end{equation}

\begin{theorem}[Equivalence of measures]\label{teonuovo}
  Let $\alpha$ and $q$ be related by \eqref{ppc33}-\eqref{ppc3d}. Then
  \begin{align}   
    \alpha\in \mathcal A &\text{ if and only if } q\in \mathcal Q , \label{pp45}\\
    \alpha\in \mathcal A^+ &\text{ if and only if } q\in \mathcal Q^+. \label{pp46}
  \end{align}
In particular, the transformations \eqref{ppc33}-\eqref{ppc3d} are one the inverse of the other and map bijectively $\mathcal A \leftrightarrow \mathcal Q$ and $\mathcal A^+ \leftrightarrow \mathcal Q^+$.
Furthermore, if $\alpha\in\mathcal A$, we have 
  \begin{align}
    \nu_{\alpha}=\varphi_q, \label{pp47}
  \end{align}
  defined in \eqref{nusola} and \eqref{ss24}.
\end{theorem}

The remaining of this subsection is devoted to the proof of Theorem \ref{teonuovo}. We start with some notation and preliminary results.
In the next three lemmas we compute the partition
function $Z_\alpha$.

Given a slot diagram $x$ we
define the \emph{translation} $\tau$ by
\begin{align}
\big(\tau x\big)_k=x_{k+1}, \qquad k=1,2,\dots.\nn
\end{align}
We have that $\tau x$ is again a slot diagram.
For $\alpha=\left(\alpha_k\right)_{k\in \mathbb N}\in \mathcal A$
we define another ``translation''  operator $\theta$  by
\begin{equation}\label{teta}
\big(\theta \alpha\big)_k:=\frac{\alpha_{k+1}}{\left(1-\alpha_1\right)^{2k}}\,, \qquad k=1,2,\dots\,,
\end{equation}
so that we can write \eqref{ppc33} as
\begin{align}
q_k= \left(\theta^{k-1}\alpha\right)_1,\quad k\ge1, \label{ppc35}
\end{align}
with the convention $\theta^0\alpha =\alpha$.
We define and compute some restricted partition functions. We call $Z_\alpha({x_k^{\infty}})$ the sum of the weights $\prod_{k\geq 1}\alpha_k^{n_k}$ over all the excursions $\vep$ such that $x_k^\infty[\vep]=x_k^\infty$.  We have 
\begin{equation}
\label{zainfi}
Z_\alpha({x_k^{\infty}}):=\prod_{n\ge k}\alpha_n^{|x_{n}|}
\sum_{\{y:y_k^{\infty}=x_k^{\infty}\}}
\prod_{\ell=1}^{k-1}\alpha_\ell^{|y_{\ell}|},\qquad Z_\alpha(x)=\prod_{n\ge 1}\alpha_n^{|x_{n}|}\,,
\end{equation}
where we sum the weights of  the slot diagrams $y$ which are compatible with $x_k^{\infty}$. These partition functions satisfy a useful recurrence:
\begin{lemma}[Iterating tail partition functions]\label{A-D}
	We have
	\begin{equation}\label{recrel}
	Z_\alpha({x_k^{\infty}})=\frac{Z_{\theta\alpha}({\left(\tau x\right)_{k-1}^{\infty}})}{(1-\alpha_1)}\,, \qquad k>1\,.
	\end{equation}
\end{lemma}
\begin{proof}
	From \eqref{zainfi} we have
	\begin{equation}\label{rec-a}
	Z_\alpha({x_k^{\infty}})=\prod_{i=k}^{\infty}\alpha_i^{|x_{i}|}
	\sum_{\left\{y_2^{\infty}:y_k^{\infty}=x_k^{\infty}\right\}}\prod_{j=2}^{k-1}\alpha_j^{|y_{j}|}
	\sum_{y_{1}\in \Z_{\ge 0}^{s_1}}\alpha_1^{|y_{1}|}\,,
	\end{equation}
	where $y_k^\ell =(y_k, y_{k+1},\dots, y_\ell)$. Note that the last sum gives $(1-\alpha_1)^{-s_1}$. If for $k<\ell<m$ we write $y_k^\ell y_{\ell+1}^m=y_k^m$, then
	\begin{equation}
	s_1=s_1\left(y_2^{k-1}x_k^{+\infty}\right)
	=1+2\sum_{i=2}^{k-1}(i-1)|y_{i}|+2\sum_{i=k}^{+\infty}(i-1)|x_{i}|\,.
	\end{equation}
	Substituting this in  \eqref{rec-a} we get
	\begin{align}
	Z_\alpha({x_k^{\infty}})&= \frac{1}{(1-\alpha_1)}\prod_{i=k}^{\infty}
	\left[\frac{\alpha_i}{(1-\alpha_1)^{2(i-1)}}\right]^{|x_i|}
	\sum_{\left\{y_2^{\infty}:y_k^{+\infty}=x_k^{+\infty}\right\}}\prod_{j=2}^{k-1}
	\left[\frac{\alpha_j}{(1-\alpha_1)^{2(j-1)}}\right]^{|y_j|}\nn\\
	& =\frac{1}{(1-\alpha_1)}\prod_{i=k-1}^{+\infty}
	\left(\theta\alpha\right)_i^{|\left(\tau x\right)_i|}
	\sum_{\left\{y_1^{\infty}:y_{k-1}^{\infty}=(\tau x)_{k-1}^{\infty}\right\}}\prod_{j=1}^{k-2}\left(\theta\alpha\right)_j^{|y_j|},
	\end{align}
	which gives \eqref{recrel}.
\end{proof}
We now compute $Z_\alpha({x_k^{\infty}})$.
\begin{lemma}[Tail partition function]\label{B-D}
	For any fixed $k\geq 2$ and $x_k^{\infty}$ we have
	\begin{align}
	\label{ss35}
	Z_\alpha({x_k^{\infty}})=\left[\prod_{i=0}^{k-2}
	\left(\frac{1}{\left(1-\left(\theta^i\alpha\right)_1\right)}\right)\right]
	\left[\prod_{j=k}^{+\infty}\left(\theta^{k-1}\alpha\right)_{j-k+1}^{|x_j|}\right]\,.
	\end{align}
\end{lemma}
\begin{proof}
	Iterating $k-1$ times the recursion \eqref{recrel} we have
	\[
	Z_\alpha({x_k^{\infty}})=\left[\prod_{i=0}^{k-2}
	\left(\frac{1}{\left(1-\left(\theta^i\alpha\right)_1\right)}\right)\right]
	Z_{\theta^{k-1}\alpha}({\left(\tau^{k-1}x\right)_1^{+\infty}})\,.
	\]
	The statement is now obtained observing that for any $x$ we have
	\[
	Z_\alpha({x_1^{+\infty}})=\prod_{i=1}^{\infty}\alpha_i^{|x_i|}\,,
	\]
	because the complete slot diagram is fixed so that there are no sums to be done.
\end{proof}
We now compute the partition function $Z_\alpha$.
Denoting by $Z^m_\alpha$ the weight of the excursions having $m$ as maximum soliton size, we have
\begin{equation}\label{ZM}
Z^m_\alpha:=\sum_{x: M(x)=m}Z_\alpha(x)\,, \qquad m\geq 0\,.
\end{equation}

\begin{lemma}[Finiteness of the partition function]\label{51a}
	The partition function $Z_\alpha$ is finite if and only if
	\begin{equation}\label{cond-sum}
	\sum_{m=0}^{+\infty}
	\left(\theta^{m}\alpha\right)_1<\infty.
      \end{equation}
      Furthermore, 
	\begin{equation}\label{ven}
	Z^m_\alpha=\left(\theta^{m-1}\alpha\right)_1\prod_{j=0}^{m-1}\left(\frac{1}{1-\left(\theta^{j}\alpha\right)_1}\right)\,, \qquad m\geq 1,
	\end{equation}
	and
	\begin{equation}\label{sab}
	Z_\alpha=1+\sum_{m=1}^{+\infty}
	\left(\theta^{m-1}\alpha\right)_1\prod_{j=0}^{m-1}
	\left(\frac{1}{1-\left(\theta^{j}\alpha\right)_1}\right)\,.
	\end{equation}
\end{lemma}

\begin{proof}
	Since  the weight of the empty excursion is 1, we have $Z^0_\alpha=1$ and \eqref{sab} is obtained from \eqref{ven} from
	the relation $Z_\alpha=\sum_{m=0}^{+\infty}Z^m_\alpha$. To show \eqref{ven} we sum over all possible slot diagrams
	\begin{align}
	& Z^m_\alpha=\sum_{x_m(0)=1}^{+\infty}\alpha_m^{x_m(0)}\sum_{\left\{x_{m-1}\in \mathbb Z_{\geq 0}^{s_{m-1}}\right\}}\alpha_{m-1}^{|x_{m-1}|}\dots
	\sum_{\left\{x_{1}\in \mathbb Z_{\geq 0}^{s_{1}}\right\}}\alpha_{1}^{|x_{1}|}\,,
	\end{align}
	where $s_k=s_k(x_{k+1}^{+\infty})$ given by \eqref{sk}-\eqref{sk2}
	with $x_k(0)=0$ for any $k>m$.
	Note that $x_m(0)$ has to be summed from $1$ up to $+\infty$ since at level $m$ there must be at least one soliton. All the other variables are summed from $0$ to $+\infty$.
	Sum on $x_1$, use \eqref{sk}-\eqref{sk2}, change name to the summed variables and iterate to obtain
	\begin{align}
	Z^m_\alpha
	&=\frac{1}{1-\alpha_1}\sum_{x_{m-1}(0)=1}^{+\infty}
	\Bigl(\frac{\alpha_m}{\left(1-\alpha_1\right)^{2(m-1)}}\Bigr)^{x_{m-1}(0)}\dots
	\sum_{\left\{x_{1}\in \mathbb N^{s_{1}}\right\}}\Bigl(\frac{\alpha_2}{\left(1-\alpha_1\right)^{2}}\Bigr)^{|x_{1}|}\\
	& =\frac{Z^{m-1}_{\theta\alpha}}{1-\alpha_1}=\dots = Z^1_{\theta^{m-1}\alpha}
	\prod_{l=0}^{m-2}\left(\frac{1}{1-\left(\theta^{l}\alpha\right)_1}\right)\,.
	\end{align}
	Hence \eqref{ven} follows from
	$$
	Z^1_{\alpha}=\sum_{x_{1}(0)=1}^{+\infty}\alpha_1^{x_{1}(0)}=\frac{\alpha_1}{1-\alpha_1}.
	$$
	It remains to discuss the convergence. We use that if $0<\beta_m<1$ then $\sum_m \beta_m<+\infty$ if and only if $\prod_m(1-\beta_m)>0$. When \eqref{cond-sum} is satisfied
	the generic term in \eqref{sab} is the product of a term of a converging series times a term converging to a finite value and therefore the series in \eqref{sab} is converging. While instead when condition \eqref{cond-sum} is violated
	the generic term in the series in \eqref{sab} is the product of a term of a diverging series times a diverging term  and therefore the series in \eqref{sab} is diverging.
\end{proof}

\begin{proof}[Proof of Theorem \ref{teonuovo}]
Consider $\alpha\in \mathcal A$ and $q=q(\alpha)$. By \eqref{ppc35} we have
\begin{equation}\label{consumq}
\sum_{k\ge1}q_k = \sum_{k\ge1}(\theta^{k-1}\alpha)_1 <\infty,
\end{equation}
by \eqref{cond-sum}. This proves $q\in\mathcal Q$.

Substituting  \eqref{ppc35} into \eqref{ven} and \eqref{sab}, we get
		\begin{equation}\label{venq}
		Z^m_\alpha=q_m\prod_{j=1}^{m}\left(\frac{1}{1-q_j}\right)\,, \qquad m\geq 1
		\end{equation}
		and
		\begin{equation}\label{sabq}
		Z_\alpha=1+\sum_{m=1}^{+\infty}
		q_m\prod_{j=1}^{m}
		\left(\frac{1}{1-q_j}\right)\,.
		\end{equation}

Under condition \eqref{consumq} the measure $\big(q_m\prod_{\ell>m}(1-q_\ell)\big)_{m\ge0}$ (recall that $q_0=1$) is a probability in $\mathbb Z_{\geq 0}$; multiplying therefore \eqref{sabq} by $\prod_{k\ge1}(1-q_k)$ we have
		\begin{align}
		Z_\alpha\prod_{k\ge1}(1-q_k) &= \prod_{k\ge1}(1-q_k) +\sum_{m\ge 1} q_m\prod_{k>m}(1-q_k)=1.\nn
		\end{align}
		This gives the alternative useful representation
		\begin{equation}\label{alt}
		Z_\alpha= 
                \prod_{k\geq 1}(1-q_k)^{-1}\,.
		\end{equation}	
		We will show that $\nu_\alpha$ satisfies the following identities.
		\begin{align}
		\nu_\alpha\left(M(x)=m\right)&
		= q_m(\alpha)\,\prod_{\ell>m}\big(1-q_\ell(\alpha)),    \quad m\ge 0,                        
		\label{ss36a}\\[2mm]
		\nu_\alpha\big(x_m(0)\big|M(x)=m\big)&= q_m^{|x_m(0)|-1}(\alpha) (1-q_m(\alpha)),  \label{ss37a}\\[2mm]
		\nu_\alpha\big(x_k\big|x_{k+1}^{\infty}\big)&= q_k^{|x_k|}(\alpha) (1-q_k(\alpha))^{s_k(x)}, \label{ss38a}
		\end{align}
Since these are the identities \eqref{ss36}-\eqref{ss38} characterizing $\varphi_q$,  \eqref{ss36a}-\eqref{ss38a} imply $\nu_\alpha=\varphi_{q}$.  
		 
		 By definition we have
		\begin{align}
		\nu_\alpha\big(M(x)=m\big)= \frac{Z^m_\alpha}{Z_\alpha}.
		\end{align}
		Using \eqref{venq} and \eqref{alt} 
		we get  \eqref{ss36a}.
		Again by definition we have
		\begin{align}
		\nu_\alpha\big(x_k\big|x_{k+1}^{+\infty}\big)=\frac{Z_\alpha({x_k^{\infty}})}
		{Z_\alpha({x_{k+1}^{\infty}})}\,.
		\end{align}
		Using \eqref{ss35} and observing that
		\begin{align}
		\frac{\left(\theta^{k-1}\alpha\right)_{i+1}}{\left(\theta^{k}\alpha\right)_{i}}
		=\big(1-(\theta^{k-1}\alpha)_1\big)^{2i}
		\end{align}
		we obtain directly \eqref{ss37a}, \eqref{ss38a}. This proves  $\nu_\alpha=\varphi_{q(\alpha)}$.

\smallskip

Conversely, assume $q\in\mathcal Q$. Then $\prod_{\ell\ge1}(1-\q_\ell)>0$ and we have
\begin{align}
\varphi_q(x)&=\Bigl(\prod_{\ell>M}(1-\q_\ell)\Bigr)\prod_{k=1}^{M}\q_k^{|x_{k}|}(1-\q_k)^{s_k} \label{ppc}\\
            &=\Bigl(\prod_{\ell\ge1}(1-\q_\ell)\Bigr) \prod_{k=1}^{M}\q_k^{|x_{k}|}(1-\q_k)^{s_k-1}\,. \label{ppc21}\\
             & = \Bigl(\prod_{\ell\ge1}(1-\q_\ell)\Bigr)\prod_{k=1}^{M}\Big[\q_k\prod_{\ell=1}^{k-1}(1-\q_\ell)^{2(k-\ell)}\Big]^{|x_{k}|}\,.\label{ppc2}
\end{align}
where we used $  s_\ell = 1+ 2\sum_{k>\ell} (k-\ell) |x_k|$. 
Comparing this expression with \eqref{nusolax}, we get $Z_{\alpha(q)} = (\prod_{\ell\ge1}(1-\q_\ell))^{-1}<\infty$ and $\varphi_q=\nu_{\alpha(q)}$, using \eqref{ppc3d}.

\smallskip

It remains to prove \eqref{pp46}. It suffices to show that $q\in\mathcal Q^+$ if and only if the mean excursion lenght under $\varphi_q$ has finite expectation.

Since by \eqref{sk2} the value of $s_k$ depends just on $|x_\ell|$ with $\ell>k$
and by the property \eqref{ss38} of the measure $\varphi_q$ we can apply Wald Theorem getting
\begin{equation}\label{wimbledon}
E_{\varphi_q}\left(|x_k|\right)=E_{\varphi_q}\left(\sum_{j=1}^{s_k}x_k(j)\right)=
m_k\beta_k,
\end{equation}
where
\begin{align}
  m_k&:=\frac{q_k}{1-q_k}\quad \text{is the mean of a Geometric$(1-q_k)$ random variable \eqref{geomm}}\nn\\
  \beta_k&:=E_{\varphi_q}\left(s_k\right), \;k\ge 1. \nn
\end{align}
By definition we have that $\beta_k\geq 1$.

The mean excursion size under $\varphi_q$ is therefore given by
\begin{equation}\label{meanlq}
E_{\varphi_q}\left(\sum_{k=1}^{\infty}2k|x_k|\right)=\sum_{k=1}^\infty2km_k\beta_k\,.
\end{equation}
Consider relationship \eqref{sk2} and take expected value with respect to the measure $\varphi_q$ on both sides. Using \eqref{wimbledon} we get the recurrence
\begin{equation}\label{rer}
\beta_k=1+\sum_{\ell>k}2(\ell-k)m_\ell\beta_\ell\,, \qquad k=0,1,2,\dots\,.
\end{equation}
where we observe that when the system has a solution, $\beta_0$ is 1 plus the mean excursion size under $\varphi_q$. 
Section 3.3 of \cite{FNRW} shows that if $\sum_k km_k<+\infty$, then the recursion \eqref{rer} has a unique finite solution $(\beta_k)_{k\geq 0}$. Since $\sum_k km_k<+\infty$ is equivalent to $\sum_k kq_k<\infty$, 
we have proven that if $q\in \mathcal Q^+$ then the mean excursion size under $\varphi_q$ is finite, which in turn implies $\alpha(q)\in\cA^+$.

Conversely, if the mean excursion size under $\varphi_q$ is finite, then the series on the right hand side of \eqref{meanlq} is convergent. Since  $\beta_k\geq 1$ this implies that $\sum_k kq_k<\infty$ holds. 
\end{proof}

\begin{remark}[$\alpha\in\mathcal A$ is a local property] \rm
We point out that while the sets $\mathcal Q$ and $\mathcal Q^+$ are identified just by asymptotic properties of the parameters $q$ (i.e. changing the values of a finite number of them does not change the belonging or not to these sets), this is not the case for the sets $\mathcal A$ and $\mathcal A^+$. For example, consider $\alpha=(\alpha_1,\alpha_2,0,\dots)$ such that $\alpha_k=0$ for any $k>2$. Then also $q_k=0$ for any $k>2$ and the partition function can be explicitly computed. 
Using \eqref{alt} and \eqref{ppc33} we obtain
\begin{equation}\label{partfin}
Z_\alpha=\frac{1-\alpha_1}{(1-\alpha_1)^2-\alpha_2}\,.
\end{equation}
We have that \eqref{partfin} is finite and positive if and only if $0\leq\alpha_1<1$ and $0\leq \alpha_2\leq (1-\alpha_1)^2$. A similar but more involved computation can be done for any finite numbers of $\alpha$'s different from zero. Notice that all the constraints on the parameters $\alpha$ are also important in order that $0\leq 0\leq (\theta \alpha)_k<1$ in definition \eqref{teta}.
\end{remark}

\subsection{Random walks and Markov chains}
We apply Theorem \ref{teonuovo} to Bernoulli product measures and Markov chains to show that those measures as seen from a record have independent components and, as corollary, that they are $T$-invariant.
\begin{lemma}[Random walks]
Consider the law of an excursion of a random walk which moves upwards with probability $\lambda\in [0,1/2)$ and downwards with probability $1-\lambda$. This measure corresponds to $\nu_\alpha$ with $\alpha\in \mathcal A^+$ and given by
 \begin{align}
   \alpha_k &= \left(\lambda(1-\lambda)\right)^k\,, \qquad k\ge1, \label{aQ0}\\
   Z_\alpha&=(1-\lambda)^{-1}. \label{zalfa}
\end{align}
\end{lemma}
\begin{proof}
	Let $\cE_n$ be the set of excursions of length $2n$. The distribution of an excursion $\vep\in\cE_n$ starting at record 0 for the random walk is $(\lambda(1-\lambda))^n(1-\lambda)$, where the last $(1-\lambda)$ is the probability to go to $-1$ at step $2n+1$. Since $\sum_{k=1}^\infty kn_k[\vep]=n$ for $\vep\in\cE_n$, we have that the random walk excursion has law $\nu_\alpha$ with parameters \eqref{aQ0}-\eqref{zalfa}.
	
	Notice that $Z_\alpha$ can also be computed when $\alpha_k=\beta^k$ for some $\beta$ as follows
	\begin{equation}\label{catcomp}
	Z_\alpha=\sum_{\vep \in \cE}\prod_{k\geq 1}\alpha_k^{n_k}=
	\sum_{n=0}^{+\infty}\sum_{\vep\in \cE_n} \beta^n= \sum_{n=0}^{+\infty}\frac{1}{n+1}\binom{2n}{n} \beta^n,
	\end{equation}
	where we used \eqref{catalan}. The last expression is the generating function of the Catalan numbers \cite{MR1676282}. Hence,
	$
	Z_\alpha=\frac{2}{1+\sqrt{1-4\beta}}= \frac1{1-\lambda}\,,
	$
	when $\beta=\lambda(1-\lambda)$. The fact that $\alpha\in \mathcal A^+$ can be verified computing
        \begin{align}
          \sum_{k\geq 1} k\rho_{k}(\alpha)=\beta \frac{\partial}{\partial \beta}\log Z_\alpha\Big|_{\beta=\lambda(1-\lambda)}= \frac{\lambda}{1-2\lambda}.
        \end{align}
        The corresponding parameters $q(\alpha)$ can be computed by \eqref{ppc33} but it seems that there is not a simple analytical expression valid for each $k$.
      \end{proof}

The following is a generalization of the previous Lemma.
Let $Q=(Q(i,j))_{i,j\in\{0,1\}}$ be the transition matrix of a Markov chain on $\{0,1\}$ and assume that the stationary probability measure $p=(p_0,p_1)$ of $Q$ satisfies $p_1\in(0,\frac12)$, that is, $Q(0,1)<Q(1,0)$.  Let  $\mu_{Q}$ be the distribution of a double infinite stationary trajectory of the Markov chain. 

\begin{lemma}[Markov chains]\label{c4e}
Consider the law of an excursion of $W\eta$ when the configuration of balls $\eta$ is distributed as $\mu_Q$.
This law corresponds to $\nu_\alpha$ with $\alpha\in \mathcal A^+$ given by
\begin{align}
\label{aQ}
\alpha_k = ab^k,\quad k\ge1\,,
\end{align}
where
\begin{align}\label{pqrel}
\left\{
\begin{array}{l}
a=Q(0,1)Q(1,0)\big[Q(1,1)Q(0,0)\big]^{-1}\,,\\
b=Q(1,1)Q(0,0)\,.
\end{array}
\right.
\end{align}
We have moreover that $Z_\alpha=1/Q(0,0)$.
\end{lemma}
\begin{proof} The probability of an excursion of the chain has a factor $(Q(0,0)Q(1,1))^{k-1}$ for each $k$-soliton, a factor $Q(0,1)Q(1,0)$ for each soliton and a global factor $Q(0,0)$ coming from the probability to go to $-1$ at the end of the excursion. That is, the probability of an excursion $\vep$ is given by
	\begin{align}
	Q(0,0) \prod_k \left(ab^k\right)^{n_k(\vep)}
	\end{align}
	that is $\nu_\alpha(\vep)$ with $\alpha$ given by \eqref{aQ} and $Z_\alpha= 1/Q(0,0)$.
	
	We can also obtain $Z_\alpha$ by summing the weights. A classic result says that the number of excursions of length $2n$ and having exactly $k$ local maxima is given by the Narayana numbers
	$$
	N(n,k)=\frac 1n\binom{n}{k}\binom{n}{k-1},
	$$
	see for example exercise 6.36 of \cite{MR1676282}. Since $\sum_{k=1}^\infty n_k$ coincides with the number of local maxima and $n=\sum_{k=1}^\infty kn_k$,
	the partition function of our Lemma is given by
	\begin{equation}\label{forZN}
	Z_\alpha=1+\sum_{n=1}^{+\infty}\sum_{k=1}^{n}N(n,k)a^kb^n=1+F(b,a)\,.
	\end{equation}
	where $F$ is the generating function of the Narayana numbers and it is known (\cite{MR1676282} exercise 6.36) to be
	$$
	F(b,a)= \frac{1-b(1+a)-\sqrt{(1-b-ba)^2-4b^2a}}{2b}\,.
	$$
	Inserting \eqref{pqrel} in \eqref{forZN} and using $Q(0,0)>Q(1,1)$ (which holds because the density is below $1/2$) we get
	$Z_\alpha= 1/Q(0,0)$ after some elementary steps.
	The fact that $\alpha\in \mathcal A^+$ can be obtained for example using \eqref{forZN} since the mean excursion size is given by
	$$
	2b\frac{\partial}{\partial b}\log\left(1+F(b,a)\right)\,.
	$$
	Here too the corresponding parameters $q(\alpha)$ can be computed by \eqref{ppc33} but it seems that there is not a simple analytical expression valid for each $k$.
\end{proof}


\section{Infinitely many balls}
\label{imb}
In this Section we consider the space of configurations with infinitely many balls, discuss the BBS dynamics, define measures on this space concatenating excursions, discuss the soliton decomposition of these measures and show that random configurations obtained by concatenating independent excursions with law $\nu_\alpha$ have independent components. As a consequence we describe a set of invariant measures with independent components. To make these statements precise we need to introduce Palm measures. 

For each  $\lambda\in[0,1]$ denote the set of configurations with density $\lambda$ by
\begin{align}
\label{eq:space}
\cX_\lambda &:=\Bigl\{\eta\in\{0,1\}^\Z: \lim_{y\to\infty}\,\frac1{y} \sum_{z=-y}^0\eta(z) =  \lim_{y\to\infty}\,\frac1{y} \sum_{z=0}^y\eta(z) =\lambda\Bigr\}, \quad \hbox{and } \nn\\
\cX&:=\cup_{0\leq\lambda<\frac12} \cX_\lambda,
\end{align}
the set of configurations with density less than $\frac12$. As we see below this space is conserved by the dynamics.

Consider a walk $\xi$. Recall that $z\in \mathbb Z$ is a \emph{record} for $\xi$ if $\xi(z)<\xi(z')$ for any $z'<z$. Notice that if $\eta\in\cX_\lambda$ and $\lambda<\frac12$, then the records have density $1-2\lambda$, as the number of empty boxes equals the number of balls between records. Denote the set of records of $\eta$ by 
$$R\eta:= \left\{z\in \mathbb Z\,:\, z\ \textrm{is \ a \ record\ of\ }W\eta\right\}$$ 
and 
$$r(\eta,i):=\min\{z\in\Z:W\eta(z)=-i\}$$ 
the position of \emph{record $i$} of the walk $W\eta$. When $\eta\in\cX$ the position $r(\eta,i)$ is well defined and belongs to $\Z$ for each $i\in \mathbb Z$. We use the notation $r(\xi,i):=r(\eta,i)$ when 
$\xi=W\eta$.

\subsection{Concatenating excursions}
\label{ce11}

Given $\eta\in\cX$ and $i\in\Z$, call $\eta^{(i)}$ the configuration between records $i$  and $i+1$ translated to the origin:
\begin{align}
  \eta^{(i)}(z)&:=\eta(r(\eta,i)+z)\,\one\{0<z<r(\eta,i+1)-r(\eta,i)\} \text{ and }\label{exc}\\
  \vep^{(i)}&:=W\eta^{(i)}\label{exc1}
\end{align}
be the corresponding walk which is indeed an excursion. The walk $\vep^{(i)}$ is called excursion $i$  of $\eta$. If $r(\eta,i+1)=r(\eta,i)+1$, we say that excursion $i$ is empty.
Again $\eta^{(i)}$ and $\vep^{(i)}$ can be considered either on a finite interval or on the whole $\mathbb Z$, since all the boxes are empty outside of a finite window.
We denote by $\underline\vep=(\vep^{(i)})_{i\in\Z}$ the collection of excursions of $\eta\in\cX$. To make the dependence on $\eta$ explicit, we may write $\vep^{(i)}[\eta]$ and $\uvep[\eta]$.

The set of configurations in $\cX$ with a record at the origin is denoted by
\begin{align}
\hcX := \{\eta\in\cX: 0\in R\eta\}.
\end{align}
Since $\cX$ has ball density less than $1/2$, if $\eta\in\hcX$ then $\eta$ has infinitely many records to the right and left of the origin, and hence, all its excursions are finite. As a consequence,
the map $\eta\mapsto \uvep[\eta]$ is a bijection between $\hcX$ and  a suitable subset of
$\cE^\Z$. The reverse map $\uvep\mapsto\eta=\eta[\uvep]$ puts record 0 of $\eta$ at the origin: $r(\eta,0)=0$ and recursively the other records using the iteration
\begin{align}\label{etavep1}
r(\eta,i+1) -r(\eta,i)&= r(\vep^{(i)},1)=2n(\vep^{(i)})+1,
\end{align}
and inserting excursion $i$ between records $i$ and $i+1$:
\begin{align}
\eta(r(\eta,i)+z) &= W^{-1}\vep^{(i)}(z), \quad 0\le z< r(\vep^{(i)},1).\label{etavep}
\end{align}
Given a configuration $\eta=\eta[\uvep]\in\hcX$, define the slots of $\eta$ by a conformal translation of the slots of the excursions:
\begin{align}
  \text{If $z$ is a $k$-slot for $\vep^{(i)}$, then $ r(\eta,i)+z$ is a $k$-slot for $\eta[\uvep]$.}
\end{align}
Since $0$ is always a $k$-slot for $\vep^{(i)}$, we have that the records of $\eta$ are also $k$-slots for all $k$.

Give label 0 to the $k$-slot at record 0: $\tts_k(\eta,0):= r(\eta,0)$, and enumerate the other $k$-slots by 
\begin{align}
  \label{pp11}
\hbox{$\tts_k(\eta,j):=$ position of the $j$-th $k$-slot, counting from $k$-slot 0, for $j\in\Z$}.
\end{align}
See Fig.~\ref{slot-enumerated-1}.

 \noindent\begin{minipage}{\linewidth}\begin{center}	
	\includegraphics[trim= 0 0 15mm 0,clip]{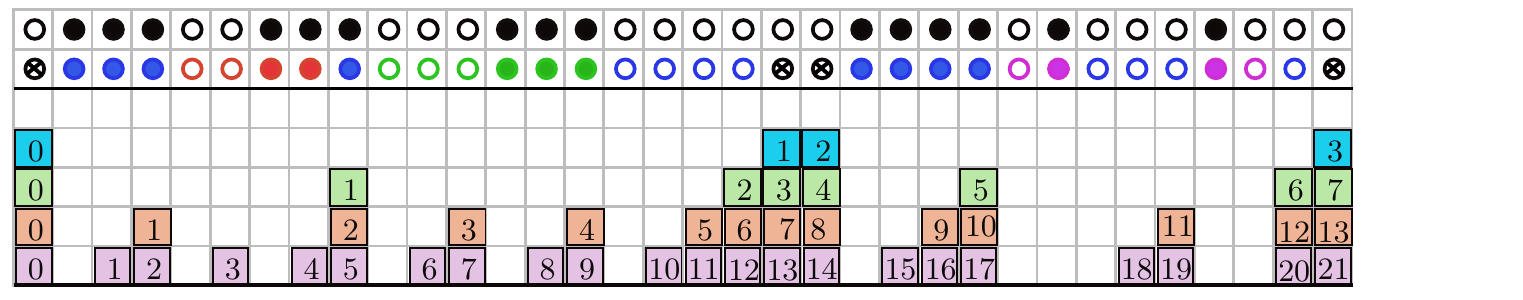}\\
	\captionof{figure}{Enumeration of slots. The upper line is a ball configuration $\eta$. The second line identifies records and solitons: records are crossed circles, 4-solitons are blue, 3-solitons green, 2-solitons red and 1-solitons purple; there are 4 records and 3 excursions, the second one being empty. In the lines below slots are represented by colored squares piled up below the corresponding box: pink, orange, green  and blue light colors correspond to 1,2,3 and 4-slots, respectively. Each $k$-slot located at the leftmost record has label 0, for each $k\ge 1$; successive $k$-slots at each line are then enumerated. The  $k$-slots labels for $k\ge4$ coincide with the labels of the records, as there are no $k$-solitons for $k>4$ in this example. 
          \label{slot-enumerated-1}}
\end{center}\end{minipage}


\paragraph{FNRW Soliton decomposition of ball configurations}  Recall the definition \eqref{appended1} and the notation \eqref{appended} where $\zeta_k(j)$ is the number of $k$-solitons appended to $k$-slot $j$.
Define $D:\hcX \to \big(\big(\Z_{\ge 0}\big)^\Z\big)^\N$ the transformation given by
\begin{align}
\label{D-def}
\eta\mapsto D\eta= \zeta =\Big(\big(\zeta_k(j)\big)_{j\in \mathbb Z}\Big)_{k\in \mathbb N}\,;
\end{align}
here $\zeta_k(j)\in \mathbb Z_{\geq 0}$.
In fact \S3.2 in \cite{FNRW} shows that $D$ is a bijection between $\hcX$ and 
\begin{align}
\cZ&:=\hbox{$\big\{\zeta\in\big(\big(\Z_{\ge 0}\big)^\Z\big)^\N: \sup\{k: \zeta_k(j)>0\}<\infty,\,\hbox{ for all }j\in\Z\big\}$}\label{zeta11}.
\end{align}
We give a construction of $D^{-1}$ in \S\ref{s5.1}. The array $\zeta=D\eta$ is called the \emph{soliton decomposition} of the configuration $\eta$. The \emph{$k$-component} of the configuration $\eta$ is  $\zeta_k=(\zeta_k(j))_{j\in \mathbb Z}$, the $k$-th row of the array $\zeta$; we also use the notation $\zeta_k=D_k\eta$.

\subsubsection{Concatenation of slot diagrams} 
Applying the Takahashi-Satsuma algorithm to each excursion $\vep^{(i)}$ we get the corresponding slot diagram $x^{(i)}=x[\vep^{(i)}]$. We can concatenate the slot diagrams to obtain the components $\zeta=D\eta$ of the configuration $\eta\in\hcX$.

The concatenation of the slot diagrams in Fig.~\ref{tantislot} is illustrated in Fig.~\ref{conca}. In Fig.~\ref{tantislot} we represent the values of the vectors of the slot diagrams inside boxes left justified. The values on line $k$ from the bottom on each diagram are the values of the vector $x^{(i)}_k$. The values on the column $j$ counting from left (and calling column $0$ the leftmost) represent the values $x^{(i)}_k(j)\,, k=1,2,\dots$. 

	
 \noindent\begin{minipage}{\linewidth}\begin{center}	
	\includegraphics{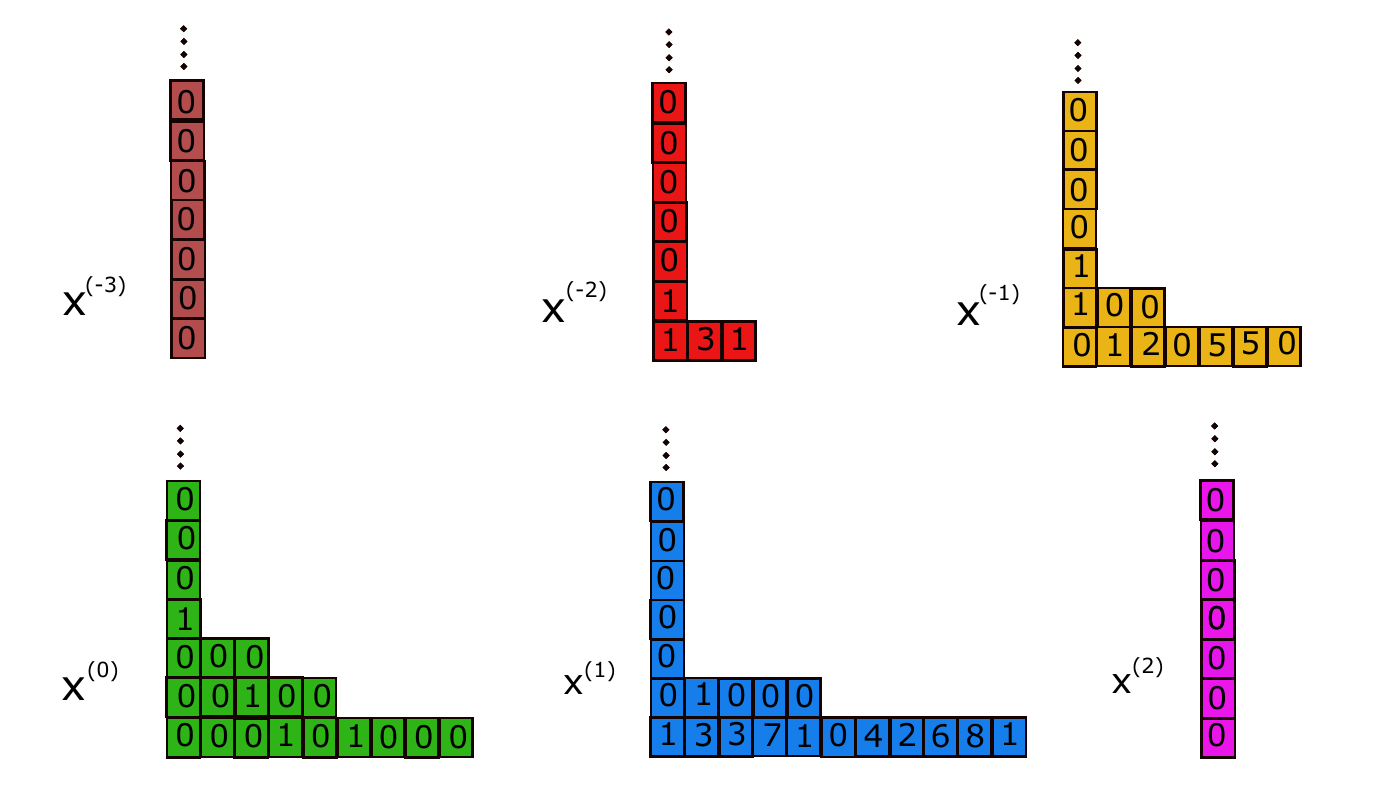}\\
	\captionof{figure}{Six different slot diagrams. Line $k$ from bottom on each diagram represents the values of the vector $x^{(i)}_k$ $i=-3,\dots 2$. Observe that  $x^{(-3)}=x^{(2)}=\emptyset$. 	\label{tantislot}}
\end{center}\end{minipage}	

The concatenation procedure is the following. The slot diagram $x^{(0)}$ maintains its shape and the column $\big(x^{(0)}_k(0)\big)_{k\geq 1}$ coincides with $\big(\zeta_k(0)\big)_{k\geq 1}$. The remaining slot diagrams are glued joining the rows of the same height in an unique row respecting the order of the labels. Boxes of the row $k$ of the slot diagram $x^{(i)}$ are to the right of the boxes of the row $k$ of the slot diagram $x^{(i-1)}$ and to the left of the boxes of the row $k$ of the slot diagram $x^{(i+1)}$. Recall that each slot diagram has an infinite column containing just zeros above the column number $0$. 

In Fig.~\ref{conca} we represent the concatenation of $6$ slot diagrams (we do not draw the infinite columns of zeros which should be drawn on the columns $-3,-2,-1,0,1,2$). Concatenating all the slot diagrams $\underline\vep[\eta]$ we obtain an infinite array such that on the column $j$ we read the values $\big(\zeta_k(j)\big)_{k\geq 1}$ of the components of the configuration $\eta$.
A formal description is given in the following paragraph.
\begin{figure}
	
	\centering
	\includegraphics{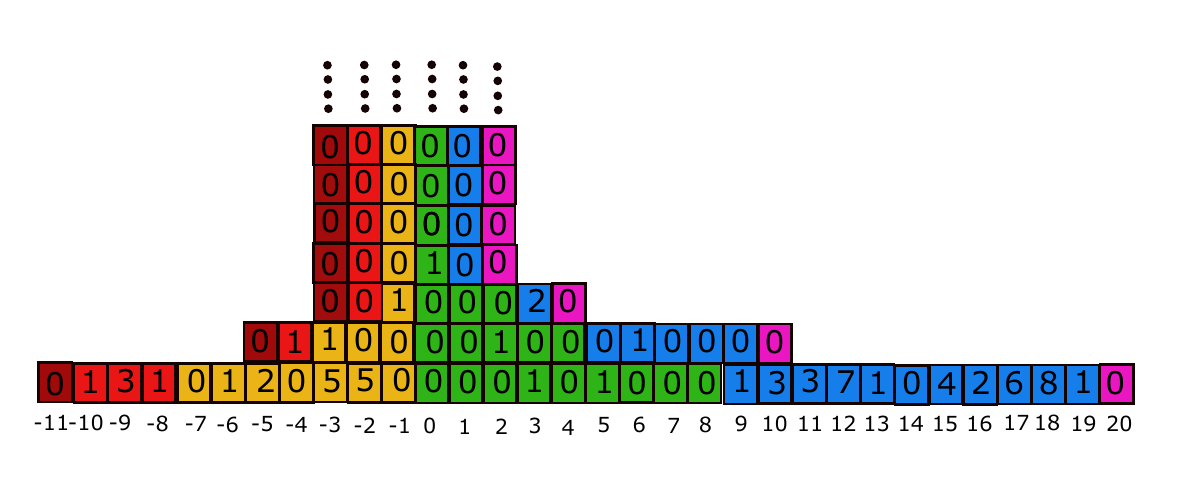}  
	
	\captionof{figure}{The result of the concatenation procedure of 6 slot diagrams. We have that $x^{(i)}\,, i=-3,-1,0,1,2$ are the ones illustrated in Fig.~\ref{tantislot} where $x^{(-3)}=x^{(2)}=\emptyset$. The integer labels below the picture represents the coordinates. On the column with label $j$ it is possible to read the values of $\big(\zeta_k(j)\big)_{k\geq 1}$.	\label{conca}	}
\end{figure}

More formally, denoting 
$s^{(i)}_k:= $ number of $k$-slots in $x^{(i)}$. Define
\begin{align}
\label{s11}
S^{(0)}_k=0;\quad S^{(i+1)}_k - S^{(i)}_k= s^{(i)}_k
\end{align}
Consider $\underline x:=\left(x^{(i)}\right)_{i\in \mathbb Z}$ and let $\zeta=\zeta[\ux]$ be defined by
\begin{align}
\label{s12}
\zeta_k(S_k^i + j) = x_k^{(i)}(j),\quad j=0,\dots,s^{(i)}_k-1\,;\ k\geq 1\,;\ i\in \mathbb Z\,.
\end{align}
It is not hard to see that the $\zeta$ so constructed is the decomposition of the configuration $\eta$ whose excursions have slot diagrams $x^{(i)}$:
\begin{align}
\label{s22}
\zeta\Big(\ux\big[\underline \vep [\eta]\big]\Big)=D\eta,
\end{align}
where $\ux\big[\underline \vep [\eta]\big]$ denotes the slot diagrams $\left(x\big[\vep^{(i)} [\eta]\big]\right)_{i\in \mathbb Z}$.
\subsubsection{From components to slot diagrams}
\label{s5.1} We explain now how to construct a family of slot diagrams starting from an array $\zeta\in\cZ$, that is, with the property $\sup\{k\ge 0: \zeta_k(j)\}<\infty$ for all $j\in\Z$.
In Fig.~\ref{fromzetatoslot} we show a portion of the infinite array $\zeta$ and discuss how to generate the slot diagrams $x^{(i)}$, $i\geq 0$. 
 \begin{figure}
	
	\centering
	\includegraphics{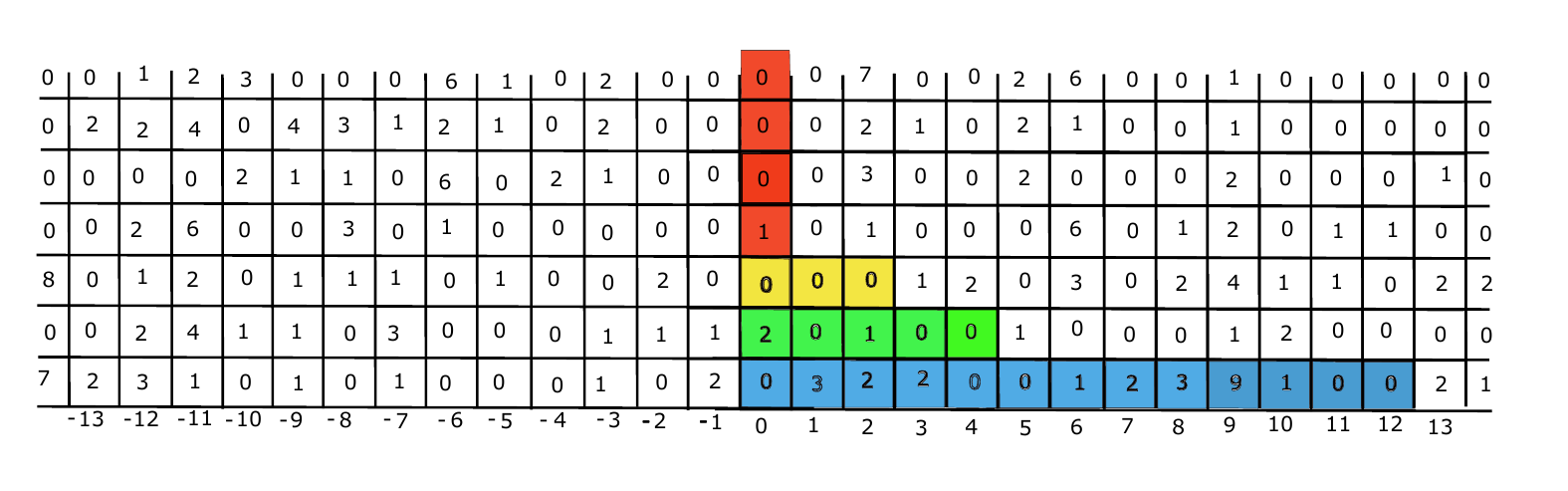}	\includegraphics{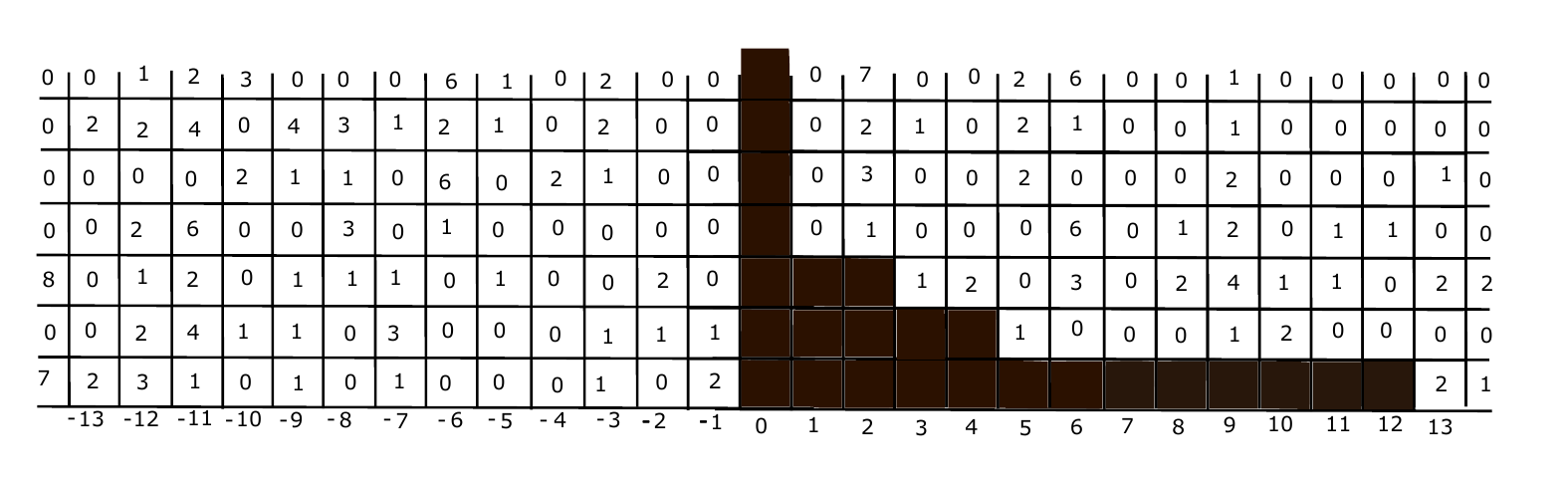}
	\includegraphics{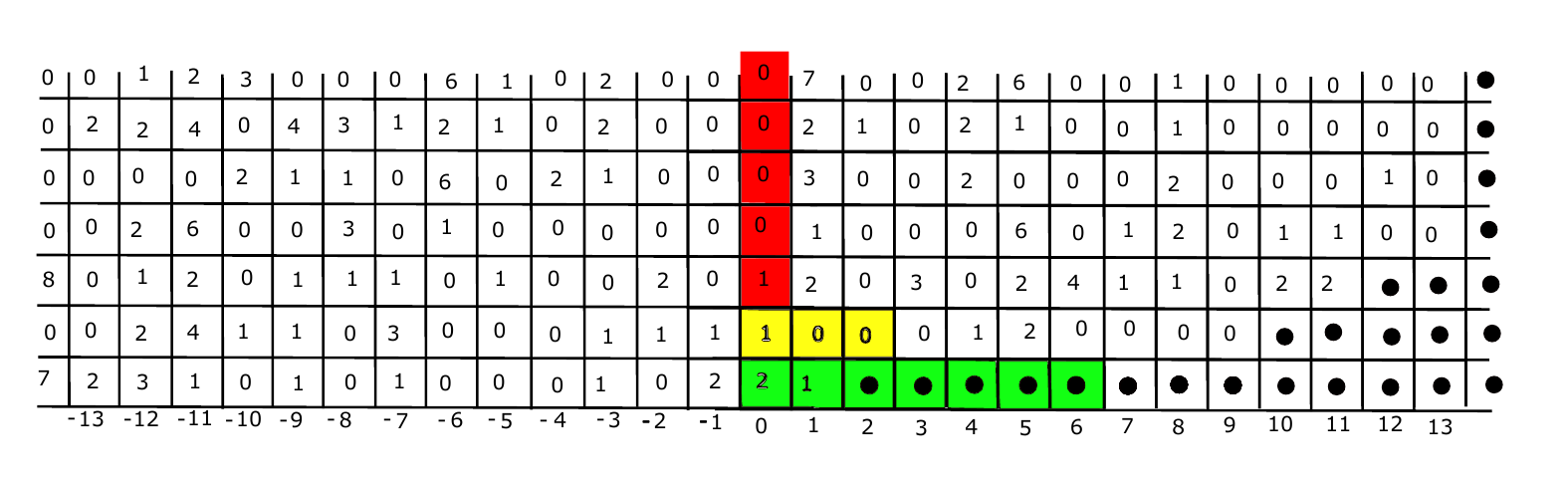}
	
	\captionof{figure}{Construction of the slot diagrams $x^{(i)}$, $i=0,1$, from an infinite array $\zeta$. In the first picture from above the slot diagram $x^{(0)}$ consists of the colored region. The squares has been added to the diagram in this order: red, yellow, green and blue. In the middle picture the squares corresponding to $x^{(0)}$ are black colored and then removed. In the bottom figure the lines have been shifted to the left to fill the empty spaces. Squares previouly outside of the picture are drawn with a $\bullet$ inside. The slot diagram $x^{(1)}$ corresponds now to the colored region with squares added following the same order as before. \label{fromzetatoslot}	}
\end{figure}

In the first step (top picture) we search for the maximal row in column $0$ such that the corresponding value is strictly positive. We color by red the square, add it to the slot diagram $x^{(0)}$ and set $M\left(x^{(0)}\right)=4$. Then we compute $s_3^{(0)}$ using \eqref{sk2}, color by yellow a corresponding number of squares in the row $3$ and add them to the slot diagram $x^{(0)}$. Now we compute $s_2^{(0)}$ again using \eqref{sk2}, color a corresponding number of squares in the row $2$ by green and add them to the slot diagram $x^{(0)}$. Finally compute $s_1^{(0)}$, color by blue a corresponding number of squares on the first line and add them to the slot diagram $x^{(0)}$. The final slot diagram number zero $x^{(0)}$ consists of all the colored region.

To construct slot diagram number 1, erase all colored boxes and shift the non erased region of each positive row to the left, until we have again an array. This is illustrated in the middle and bottom picture of Fig.~\ref{fromzetatoslot}. In the middle picture we colored black those squares to be deleted, while in the bottom picture we shifted the lines to the left. Each line has been shifted by the corresponding number of $\bullet$ appearing on the right. Apply now the algorithm we have used above to identify slot diagram zero and call the result slot-diagram 1. This is illustrated again in the bottom picture of Fig.~\ref{fromzetatoslot} using the same order of the colors. Repeat the procedure to construct the slot diagrams with nonnegative label.

To construct the slot diagrams with negative label, use the same algorithm as for label zero but working from right to left, and the procedure illustrated in Fig.~\ref{fromzetatoslot-meno}.  

\noindent\begin{minipage}{\linewidth}\begin{center}
		\includegraphics{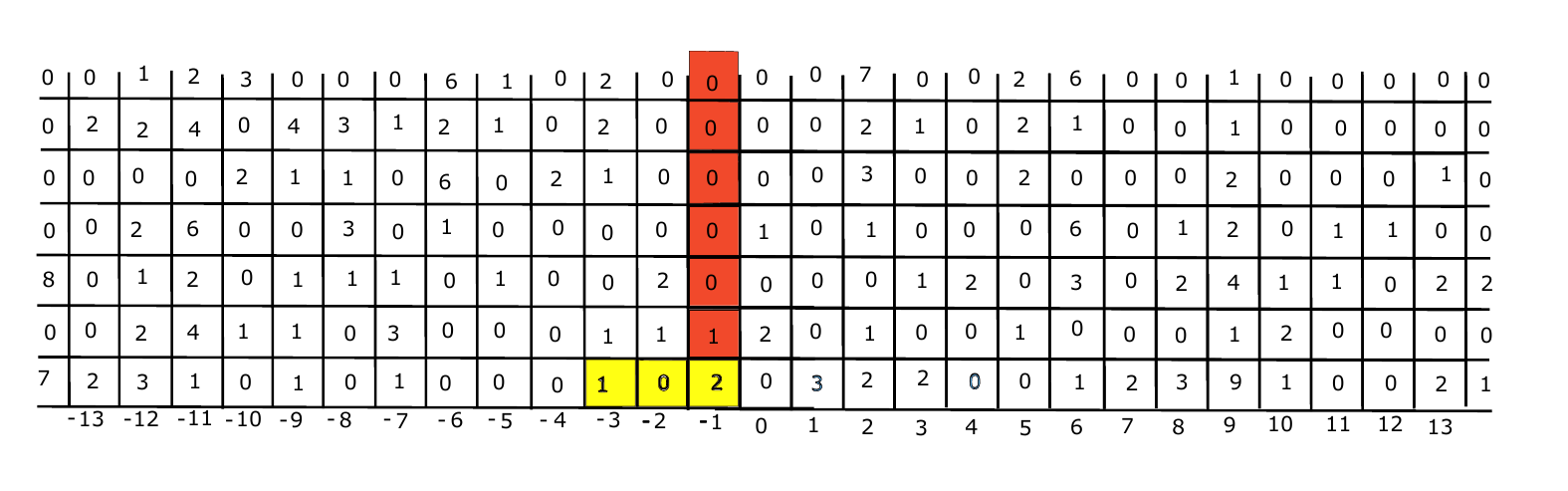}	\includegraphics{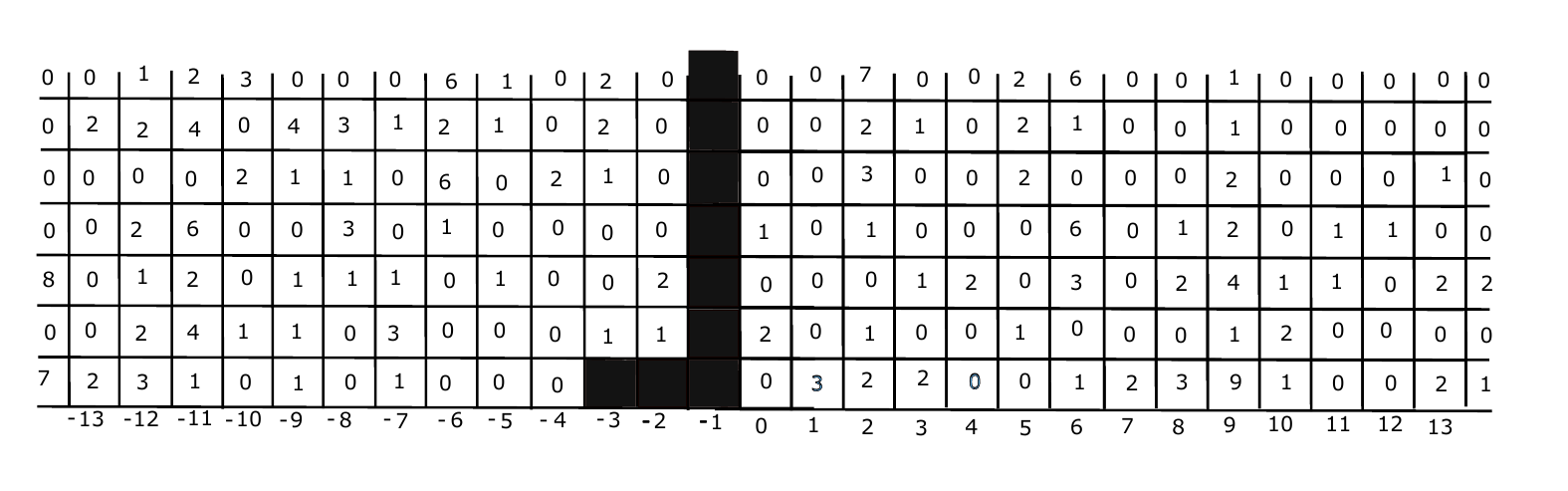}
		\includegraphics{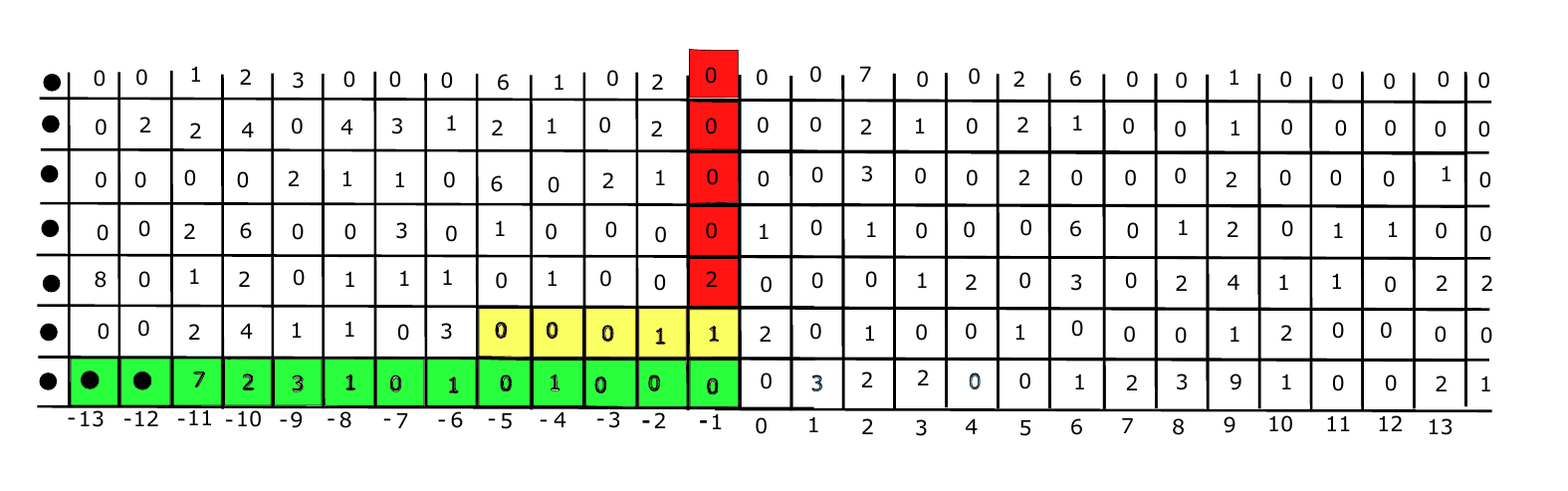}
		
		\captionof{figure}{Construction of the slot diagrams $x^{(i)}$, $i=-1,-2$ from the infinite array $\zeta$. The  slot diagrams consist of the colored region added according the same rules as before. On the top picture we have $x^{(-1)}$, on the middle picture we color black the squares to be removed, on the bottom picture we shift to the right the rows and construct $x^{(-2)}$.\label{fromzetatoslot-meno}	}
\end{center}\end{minipage}

Finally the slot diagrams produced by the above iterations of the algorithm are the following
\begin{figure}
	\centering

		\includegraphics{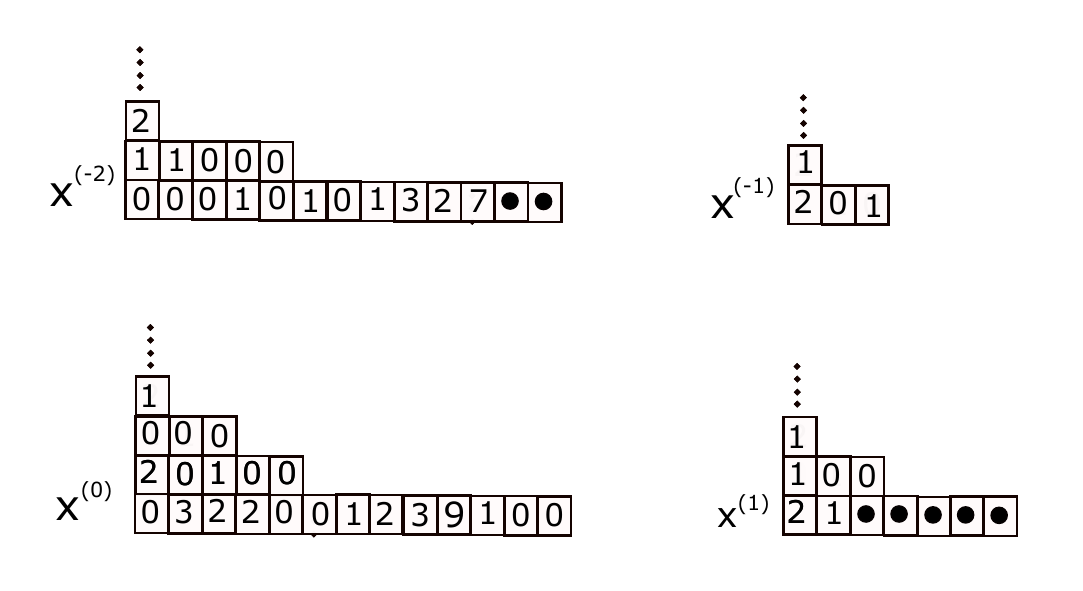}	
		
		\captionof{figure}{The four slot diagrams produced by the iteration of the algorithm in Figures \ref{fromzetatoslot} and \ref{fromzetatoslot-meno}. \label{tanti-slot-dopo}	}

\end{figure}


This construction is formally described as follows. Let $\zeta =\big((\zeta _{k}(j))_{j\in\Z}\big)_{k\geq 1}$ belong to $\cZ$. We construct a slot-diagram $x=x[\zeta]$ as follows. Set
\begin{align}
  \label{m<infty}
  M(x) := \sup\{k\ge 0: \zeta _{k}(0)>0\}<+\infty\,,
\end{align}
a bounded nonnegative integer. Call $m=M(x)$ and set
\begin{align}
  s_k&=1, \hbox{ for }k\ge m
    ,\;  \nn\\
  x_k(0)&= 0, \hbox{ for }k> m, \nn\\
  x_m(0)&= \zeta _m(0).
\end{align}
Assume $\big(x_\ell(0),\dots,x_\ell(s_\ell-1)\big)$ is known for $\ell>k$ and iteratively define
\begin{align}
     |x_\ell|&= \sum_{j=0}^{s_\ell-1} x_\ell(j),\nn\\
  s_k &= 1+ 2\sum_{\ell>k} (\ell-k) |x_\ell|,\nn\\
  x_k(j) &= \zeta _k(j), \quad j=0,\dots,s_k-1.
\end{align}
We have constructed a slot diagram
\begin{align}
  \label{xki}
  x:=\big(x_k(0),\dots,x_k(s_k-1)\big)_{k\ge1}.
\end{align}
Write $x[\zeta]$ and $s_k[\zeta]$ to stress that $x$ and $s_k$ are functions of $\zeta$ and define the hierarchical translation
\begin{align}\label{tr}
   \phi \zeta  = (\tau^{s_k[\zeta]} \zeta _k)_{k\ge1}\,.
\end{align}
The coordinate $s_k[\zeta]$ is the leftmost positive coordinate of $\zeta_k$ not used in the construction of $x[\zeta ]$. We stress that the translation $\tau^{s_k[\zeta]}$ in \eqref{tr} acts on the index labeling the slots, more precisely
$$
(\phi \zeta )_k(j) =\tau^{s_k[\zeta]} \zeta _k(j)=\zeta_k(j+s_k[\zeta])\,.
$$
Since $s_k=1$ for all $k\ge m$, we have  $(\phi\zeta)_k(j)=\zeta_k(j+1)$ for all $k\ge m$. Hence, since  $\zeta $  belongs to the set \eqref{zeta11}, so does $\phi\zeta$ and we can define iteratively
\begin{align}
  \label{xi68}
  x^{(i)} :=    x[\phi^i\zeta ], \quad i\ge0.
\end{align}
For negative $i$ let $\zeta '$ be the reflection of $\zeta $ with respect to the origin translated by $-1$: $\zeta '(j):=\zeta (-j-1)$ for $j\in\Z$ and define
\begin{align}\label{rsd}
  x^{(i)} := (x[\phi^{-i-1}\zeta'])', \quad i<0\,,
\end{align}
that is, construct the slots diagrams for $\zeta '$, reflect the obtained slot diagrams, assign the reflected slot diagram of $0$ to $-1$ and so on. In \eqref{rsd} for a slot diagram $x$ we defined the reflected one $x'$ by $x'_k(j)=x_k(s_k-j-1)$. The corresponding excursions are then given by
\begin{align}
  \vep^{(i)}:= \vep[x^{(i)}],\quad \uvep=(\vep^{(i)})_{i\in\Z}.
\end{align}
\begin{lemma}[FNRW]
  \label{d-1}
  The configuration $\eta=\eta[\uvep]$ satisfies $D\eta=\zeta $.
\end{lemma}
See \S2.3, ``Reconstructing the configuration from the components'' in \cite{FNRW} for a proof of this Lemma. This implies that $D$ is a bijection between $\cup_{\lambda < 1/2}\hcX_\lambda$ and $\cZ$ and we can write $\eta=D^{-1}\zeta$.

\subsection{Measures on ball configurations and soliton components}
\label{ac11}
We here define distributions on arrays of components and ball configurations starting with independent families of iid excursions and vice-versa.

\paragraph{Palm measures} We consider configurations with all records and the underlying point process of the records. Start reminding the definition of Palm and anti-Palm measures, see Chapter 8 of Thorisson \cite{thorisson} for background and proofs of the following facts, which are stated with respect to the point process of the records.

Let $\mu$ be a translation invariant measure on $\cX$ and define $\lambda=\lambda(\mu):=\int \eta(0)\,\mu(\dd\eta)$ its mean density; the density of records is then $1-2\lambda$. Define the measure $\Palm(\mu)$ on $\hcX$ by acting on test functions $f$ by
\begin{align}
  \int f(\eta)\Palm(\mu)(\dd\eta) = \frac1{1-2\lambda} \int \one\{0\in R\eta\} f(\eta)\mu(\dd\eta).
\end{align}
This is the measure $\mu$ conditioned to have a record at the origin. $\Palm(\mu)$ is record-translation invariant:
\begin{align}
  \int f(\eta)\Palm(\mu)(\dd\eta)= \int f(\tau^{r(\eta,i)}\eta)\Palm(\mu)(\dd\eta),\quad\text{for all record }i.
\end{align}

Reciprocally, for a record-translation invariant measure $\hmu$ on $\hcX$ with finite average inter-record distance
\begin{align}\label{cappa}
 \kappa(\hmu)  :=\int r(\eta,1) \hmu(\dd\eta) \in[1,\infty),
\end{align}
define the \emph{anti-Palm} measure $\Palm^{-1}(\hmu)$ acting on test functions $f$ as
\begin{align}
  \label{anti-palm}
 \Palm^{-1}(\hmu) f:= \frac1{\kappa(\hmu)}\int \sum_{z=1}^{r(\eta,1)} f(\tau^z\eta)\, \hmu(\dd\eta).
\end{align}
The measure $\mu:=\Palm^{-1}(\hmu)$ is translation invariant and has mean ball density
\begin{equation}\label{lambda}
  \lambda(\mu) = \frac{\kappa(\hmu)-1}{2\kappa(\hmu)}\in[0,\textstyle{\frac12})\,,
\end{equation}
indeed $\frac12 (\kappa(\hmu)-1)$ is the mean number of balls per excursion, that is, between two successive records  and $\kappa(\hmu)$ is the mean distance between successive records. 
There are record-translation invariant measures $\hmu$ with infinite average inter-record distance, but concentrating on the set of configurations with all records finite. The anti-Palm transformation of those measures is not defined.

The next proposition proven by FNRW says that random arrays in $\cZ$ with translation invariant distribution and independent components produce record-translation invariant distributions on the space of ball configurations. 
\begin{proposition}[FNRW, Independent components and Palm measures]
  \label{p88}
  Let $\zeta$ be a random array with translation invariant distribution concentrating on $\cZ$ and satisfying $(\zeta_k)_{k\ge 1}$ independent. Then the law of $D^{-1}\zeta$, denoted by $\hmu$, is record-translation invariant. Furthermore, if $\sum_k kE\left[\zeta_k(0)\right]<\infty$, then the inter-record distance under $\hmu$ is finite and the measure $\Palm^{-1}(\hmu)$ is translation invariant and concentrates on $\cX$.
\end{proposition}

We have the following result.

\begin{theorem}[Soliton weights and independent geometric components]\label{teo9}

a) Let $\alpha\in\cA$ and $\uvep=(\vep^{(i)})_{i\in\Z}$ be a sequence of i.i.d.~random excursions with distribution $\nusola$ given by~\eqref{nusola}. Let $\hmu_\alpha$ be the distribution of $\eta=\eta[\uvep]$, the random ball configuration with Record~0 at the origin and excursions $(\vep^{(i)})_{i\in\Z}$, defined in~\eqref{etavep}.  Define $\zeta:= D\eta$, the soliton decomposition of $\eta$, defined in~\eqref{D-def}. Then $\zeta_k(j)$ are independent Geometric$(1-q_k(\alpha))$ random variables, for all $j\in\Z, k\ge1$. 

b) Reciprocally, let $q\in\cQ$ and $\zeta= \big((\zeta_k(j))_{j\in\Z}\big)_{k\ge 1}$ be an array of independent random variables with $\zeta_k(j)$ distributed according to Geometric$(1-q_k)$, for all $j\in\Z$ for all $k\ge 1$. Then $\zeta\in\cZ$ with probability 1 and denoting $\eta:=D^{-1}\zeta$, we have that $\vep^{(i)}[\eta]$ are i.i.d.~excursions with law  $\nu_{\alpha(q)}$, so that $\eta$ has law $\hmu_{\alpha(q)}$, a record-translation invariant measure. 

\end{theorem}
\begin{proof}
  a)  Let $x^{(i)}= x[\vep^{(i)}]$ be the slot diagram of the excursion $\vep^{(i)}$. By Theorem~\ref{teonuovo} $x^{(i)}$ satisfies \eqref{ss37} and \eqref{ss38}, that is, given the number of $k$-slots $s^{(i)}_k$, the variables $x^{(i)}_k(0),\dots, x^{(i)}_k(s^{(i)}_k-1)$ are i.i.d.~Geometric$(1-q_k(\alpha))$ random variables. Let $\cF_k$ be the sigma field generated by the $k$-th row $(x^{(i)}_k)_{i\in\Z}$ and denote by $\cF_{>k}$ the sigma field generated by $\big((x^{(i)}_{k+1})_{i\in\Z},(x^{(i)}_{k+2})_{i\in\Z},\dots\big)$, the rows bigger than $k$. Condition on $\cF_{>k}$ and construct $\zeta_k$ using \eqref{s12}, that is juxtaposing the $k$-component of each slot diagram one after the other. Since the excursions are independent, the resulting component $\zeta_k\in(\Z_{\ge0})^\Z$ consists of i.i.d.~Geometric$(1-q_k(\alpha))$ random variables independently of the conditioning.  This implies that $\zeta_k(j)$ are independent Geometric$(1-q_k(\alpha))$ random variables, for all $j\in\Z, k\ge1$, concluding the proof of item a.

b)  It suffices to show that the excursions generated by the slot diagrams $(x^{(i)}[\zeta])_{i\in\Z}$ have marginal law $\nu_{\alpha(q)}$ and are independent. It is immediate from the construction illustrated in Fig.~\ref{fromzetatoslot} (top picture) that $x^{(0)}[\zeta]$ satisfies \eqref{ss36}-\eqref{ss38}. Let the array $\zeta^{(1)}$ obtained by erasing the entries used by $x^{(0)}[\zeta]$ and sliding the remaining entries to the left (Fig.~\ref{fromzetatoslot}). Since the set of erased entries does not depend on de contents of the non-erased entries and the entries in $\zeta$ are independent, $\zeta^{(1)}$ has the same law as $\zeta$ and it is independent of $x^{(0)}[\zeta]$. Then $x^{(1)}[\zeta] = x^{(0)}[\zeta^{(1)}]$, which is independent of the previous slots diagrams. The same argument applies to the construction of the slot diagrams of $\zeta$ with negative label (see Fig.~\ref{fromzetatoslot-meno}). 
\end{proof}

\subsection{Invariant measures for the BBS}
Theorem \ref{t1} below, proven by  FNRW, states that a translation invariant measure whose Palm transform has independent components is invariant for the BBS dynamics. As a consequence of Theorems \ref{t1} and \ref{teo9}, we will conclude that the measure $\mu_\alpha=\Palm^{-1}(\hat\mu_\alpha)$ introduced in Theorem~\ref{teo9} is invariant for the dynamics. 

The BBS dynamics can be described by the operator $T$ acting on configurations $\eta\in\cX$ by
\begin{align}
  T\eta(z) := (1-\eta(z))\,\one\{z\notin R\eta\}.
\end{align}
The configuration $T\eta$ coincides with $\eta$ at the records of $\eta$ and the contents of the other boxes are inverted. Indeed, at each iteration of $T$ the balls in each excursion go to the empty boxes of the same excursion and the record boxes remain empty. In particular, the number of balls and empty boxes of $\eta$ and $T\eta$ between two successive records of $\eta$ are the same. Since the records have a positive density, this implies that density is conserved by $T$: $T\cX_\lambda =\cX_\lambda$ for any $\lambda\in[0,1/2)$ and that $T:\cX\to\cX$  indeed. When $\eta$ has finitely many balls $T\eta$ coincides with the configuration obtained after the carrier has visited all boxes of the configuration $\eta$, as described in the introduction.

We say that a measure $\mu$ is \emph{$T$-invariant}  if $\mu\circ T^{-1}=\mu$. The next theorem of FNRW establishes conditions under which translation invariant measures with independent soliton components are $T$-invariant.
\begin{theorem}[FNRW. Independent components and $T$-invariance]
  \label{t1}
  Let  $\zeta=(\zeta_k)_{k\ge 1}$ be a random array  with translation invariant distribution and independent rows satisfying $\sum_k kE\zeta_k(0)<\infty$. 
  Let $\hmu$ be the law of $D^{-1}\zeta$. Then $\mu:=\Palm^{-1}(\hmu)$ is $T$-invariant.
\end{theorem}

We have proven in Theorem \ref{teo9} that for $\alpha\in\cA$ the measure obtained by concatenating i.i.d.~copies of excursions with law  $\nu_\alpha$ has independent components. Applying then Theorem \ref{t1} we conclude in Theorem \ref{t2} below that if $\alpha\in \mathcal A^+$ this measure is the Palm measure of a $T$-invariant measure. As particular cases, we deduce in Corollaries \ref{c3} and \ref{c4} that product measures and stationary Markov chains in $\{0,1\}$ with density of balls less than $\frac12$ are $T$-invariant, a fact proven in \cite{FNRW} and \cite{CroydonKatoSasadaTsujimoto17} using classical arguments and reversibility properties of queues.

We now show that if $\alpha\in\cA^+$, then  $\mu_\alpha$ is $T$-invariant and that if $q\in\cQ^+$, then $\mu_{\alpha(q)}$ is $T$-invariant. When $\alpha\in \mathcal A^+$ we have
\begin{equation}\label{inpiu}
\kappa(\alpha):=\kappa(\hmu_\alpha)= 1+2\sum_k k\rho_k(\alpha)<\infty\,,
\end{equation}
where $\kappa(\hmu)$ is defined in \eqref{cappa}. We define also $\lambda(\alpha):=\lambda(\mu_\alpha)$ where $\lambda(\mu)$ is defined in \eqref{lambda}.
\begin{theorem}[$\mu_\alpha$ is $T$-invariant]\label{t2} 
	
a) Assume the conditions of Theorem \ref{teo9}.
  If  $\alpha\in\cA^+$, then $\hmu_\alpha$ concentrates on $\widehat\cX_{\lambda(\alpha)}$ and it is record-translation invariant and the measure $\mu_\alpha:=\Palm^{-1}(\hmu_\alpha)$ concentrates on $\cX_{\lambda(\alpha)}$, it is translation invariant and $T$-invariant.

\smallskip 

b) If  $q\in\cQ^+$, then $\hmu_{\alpha(q)}$ concentrates on $\widehat\cX_{\lambda(\alpha(q))}$ and $\mu_{\alpha(q)}:=\Palm^{-1}(\hmu_{\alpha(q)})$ is translation invariant, concentrates on $\cX_{\lambda(\alpha(q))}$ and it is $T$-invariant.
\end{theorem}

Next corollaries prove that product measures on $\cX$ and stationary trajectories on $\cX$ of Markov chains on $\{0,1\}$ may be expressed as $\mu_\alpha$ of Theorem \ref{t2}, by choosing the appropriate~$\alpha$ and/or $q$. In particular those measures are $T$-invariant, a fact already proven by using reversibility of those trajectories by \cite{CroydonKatoSasadaTsujimoto17} and \cite{FNRW}.

\begin{corollary}[Product measures]
  \label{c3}
  Let $\lambda\in[0,\frac12)$ and $\pi_\lambda$ be the product measure on $\cX$ with density $\lambda$. Let $\hpi_\lambda:=\Palm(\pi_\lambda)$ and $\eta$ be distributed with $\hpi_\lambda$. Define
  \begin{align}
    \label{alpha13}
     \alpha_k := \left(\lambda(1-\lambda)\right)^k.
  \end{align}
 Then  $\alpha\in\cA^+$ and the random excursions $\big(\vep^{(i)}[\eta]\big)_{i\in\Z}$ are i.i.d.~with distribution $\nusola$,
the soliton components $(D_k\eta)_{k\ge0}$ are mutually independent and  the $k$-soliton component $(D_k\eta(j))_{j\in\Z}$ is a sequence of i.i.d.~Geometric$(1-q_k(\alpha))$ random variables. As a consequence, the measure $\pi_\lambda$ is $T$-invariant.
\end{corollary}
\begin{remark}  [Mean number of solitons per site] \rm
  Denote by $\delta_k$ the mean number of $k$-solitons per site under the product measure $\pi_\lambda$. It is given by $\delta_k=\rho_k\,(1-2\lambda)$, where $\rho_k$ is the mean number of $k$-solitons per excursion, that is between successive records and $(1-2\lambda)^{-1}$ is the mean distance between successive records under $\pi_\lambda$.   Kuniba and Lyu \cite{kl}  have computed an explicit expression for $\delta_k$ in terms of $\lambda$. 
\end{remark}

Corollary \ref{c3} is a special case of the next corollary for Markov chains.
\begin{corollary}[Markov chains and Ising models]
  \label{c4}
Let $Q=(Q(i,j))_{i,j\in\{0,1\}}$ be the transition matrix of a Markov chain in $\{0,1\}$ and assume that the stationary probability measure $(p_0,p_1)$ of $Q$ satisfies $p_1\in(0,\frac12)$. Let  $\pi_{Q}$ be the distribution of a double infinite stationary trajectory of the chain. Define $\hpi_Q:=\Palm(\pi_Q)$ and $\eta$ be a configuration with law $\hpi_Q$. Define $\alpha=(\alpha_k)_{k\ge1}$ by $\alpha_k := ab^k$ for $a,b$ defined in function of $Q$ by~\eqref{pqrel}.
Then $\alpha\in \cA^+$ and the random excursions $\big(\vep^{(i)}[\eta]\big)_{i\in\Z}$ are i.i.d.~with distribution $\nusola$,
the soliton components $(D_k\eta)_{k\ge0}$ are mutually independent and the $k$-soliton component $(D_k\eta(j))_{j\in\Z}$ is a sequence of i.i.d.~Geometric$(1-q_k(\alpha))$ random variables. As a consequence, $\mu_Q$ is $T$-invariant.
\end{corollary}

\begin{remark}[Infinite expected excursion size]\label{iees} \rm
When $\alpha\in\cA\setminus\cA^+$, the mean excursion size under $\nu_\alpha$ is infinite and $\hmu_\alpha$ defined in Theorem \ref{teo9} has infinite mean inter-record distance. The independence of components is still valid in this case but $\Palm^{-1}(\hmu_\alpha)$ cannot be defined \cite{thorisson}. In particular, when $\alpha$ is given by \eqref{alpha13} with $\lambda=\frac12$, $\nu_\alpha$ is the law of an excursion of the symmetric random walk, $\hmu_\alpha$ is well defined but its inverse-Palm measure is not, as the density of records is 0. 
\end{remark}

\begin{proof}
  [Proof of Theorem \ref{t2}] 
  
  a) Since $\alpha\in\cA^+$ the mean inter-record distance is finite under $\hmu_\alpha$ \eqref{inpiu} and therefore the measure $\mu_\alpha$ is well defined and translation invariant, as we saw in \S\ref{ac11}. The fact that $\mu_\alpha$ is $T$-invariant will follow by Theorem \ref{t1} once we show that $\sum_kkE\zeta_{k}(j)<\infty$. Since $\zeta_k(j)$ is Geometric$(1-q_k(\alpha))$, the condition $\sum_kk E\zeta_k(j)<\infty$  is equivalent to $q(\alpha)\in \mathcal Q^+$ that follows by Theorem \ref{teonuovo}.
  
  \smallskip

b) Since $q\in \mathcal Q^+$ and $\zeta_k(j)$ is Geometric$(1-q_k)$ we have 
$\sum_kkE\zeta_{k}(j)<\infty$ and we can apply Theorem \ref{t1}.
\end{proof}

\section*{Acknowledgments}

We thank the referees of the Electronic Journal of Probability for their careful reading and for many helpful comments for the presentation of the results.
PAF thanks many illuminating discussions with Leo Rolla.

This project started when the first author was visiting Gran Sasso Science Institute in L'Aquila in 2016 and then developed during the stay of the authors at the Institut Henri  Poincare, Centre Emile Borel, during the trimester \emph{Stochastic Dynamics Out  of Equilibrium} in 2017. We thank both institutions for warm hospitality and support.


\begin{thebibliography}{99}








\bibitem{CroydonKatoSasadaTsujimoto17} Croydon, D. A., Kato, T., Sasada, M., and Tsujimoto, S. \emph{Dynamics of
the box-ball system with random initial conditions via Pitman's transformation.}
Preprint arXiv:1806.02147, 2018.

\bibitem{FG} Ferrari P. A., Gabrielli D. \emph{Box-ball system: soliton and tree decomposition of excursions.}  arXiv:1906.06405
 

\bibitem{FNRW} Ferrari, P. A., Nguyen, C., Rolla, L. T., and Wang, M. \emph{Soliton decomposition of the Box-Ball-System}. Preprint arXiv:1806.02798, 2018.

\bibitem{kl} Kuniba A., Lyu H. \emph{Large deviations and one-sided scaling limit of randomized multicolor box-ball system.} arXiv:1808.08074



\bibitem{levine-lyu-pike} Levine, L., Lyu, H., and Pike, J. \emph{Double jump phase transition in a soliton cellular automaton.} Preprint. arXiv:1706.05621, 2017.


\bibitem{MR1676282} Stanley, R. P. \emph{Enumerative combinatorics.} Vol. 2, vol. {\bf 62} of Cambridge Studies
in Advanced Mathematics. Cambridge University Press, Cambridge, 1999. With a
foreword by Gian-Carlo Rota and appendix 1 by Sergey Fomin.


\bibitem{TS} Takahashi, D., and Satsuma, J. \emph{A soliton cellular automaton.} Journal of The Physical Society of Japan {\bf 59}, 10 (1990), 3514-3519.


\bibitem{thorisson} Thorisson, H. \emph{Coupling, stationarity, and regeneration.} Probability and its Applications (New York). Springer-Verlag, New York, 2000.
	



 \end{thebibliography}

\end{document}